\documentclass{amsart}
\usepackage{amsmath,amscd}

\newtheorem{theorem}{Theorem}[section] 
\newtheorem{lemma}[theorem]{Lemma}

\newtheorem{proposition}[theorem]{Proposition}

\theoremstyle{definition}

\theoremstyle{remark}

\numberwithin{equation}{section}

\begin{document}

\title{The Continuous Spectrum in Discrete Series Branching Laws}

\author{Benjamin Harris}
\address{Department of Mathematics, Louisiana State University, Baton Rouge, Louisiana 70803}
\email{blharris@lsu.edu}
\thanks{The first author was an NSF VIGRE postdoc at LSU while this research was conducted.}

\author{Hongyu He}
\address{Department of Mathematics, Louisiana State University, Baton Rouge, Louisiana 70803}
\email{hongyu@math.lsu.edu}

\author{Gestur \'{O}lafsson}
\address{Department of Mathematics, Louisiana State University, Baton Rouge, Louisiana 70803}
\email{olafsson@math.lsu.edu}
\thanks{The third author was supported by NSF grant 1101337 while this research was conducted.}


\date{September 10, 2012}



\begin{abstract}
If $G$ is a reductive Lie group of Harish-Chandra class, $H$ is a symmetric subgroup, and $\pi$ is a discrete series representation of $G$, the authors give a condition on the pair $(G,H)$ which guarantees that the direct integral decomposition of $\pi|_H$ contains each irreducible representation of $H$ with finite multiplicity. In addition, if $G$ is a reductive Lie group of Harish-Chandra class, and $H\subset G$ is a closed, reductive subgroup of Harish-Chandra class, the authors show that the multiplicity function in the direct integral decomposition of $\pi|_H$ is constant along `continuous parameters'. In obtaining these results, the authors develop a new technique for studying multiplicities in the restriction $\pi|_H$ via convolution with Harish-Chandra characters. This technique has the advantage of being useful for studying the continuous spectrum as well as the discrete spectrum.
\end{abstract}

\maketitle

\section{Introduction}

If $\pi$ is an irreducible representation of a Lie group $G$ and $H\subset G$ is a type 1 closed Lie subgroup, it is natural to consider the restriction $\pi|_H$ and to decompose this restriction into irreducible representations of $H$. This type of problem comes up naturally in physics; see for instance \cite{Ok}, \cite{RWB}. It also arises naturally in number theory; see for instance section 9 of \cite{OS}, \cite{B}, or \cite{KOd}. In the first case, the multiplicities in this decomposition should correspond to physical quantities and in the second case they should correspond to dimensions of spaces of automorphic forms. Either way, one observes that many branching laws of particular significance have the property that $\pi|_H$ has finite multiplicities.

In \cite{Ko1} and \cite{Ko2}, Kobayashi gives a sufficient condition for an irreducible representation of a real reductive Lie group $G$ of type $\overline{A_{\mathfrak{q}}(\lambda)}$ to restrict discretely to a closed reductive subgroup of Harish-Chandra class, $H$. In \cite{Ko3}, he gives a necessary condition for infinitesimal discrete decomposability, and he shows that the necessary and sufficient conditions agree when $(G,H)$ is a symmetric pair. In addition, he shows that when $(G,H)$ is a symmetric pair and the restriction is infinitesimally discretely decomposable, the branching law $\overline{A_{\mathfrak{q}}(\lambda)}|_H$ contains each irreducible $H$ representation with finite multiplicity. More recently, it has been shown that if $(G,H)$ is a strong Gelfand pair of real, reductive algebraic groups and $\pi$ is an irreducible, unitary representation of $G$, then the discrete spectrum of $\pi|_H$ contains each irreducible representation of $H$ with multiplicity at most one (see \cite{SZ} and the references therein). Finally, in the recent preprint \cite{KO}, Kobayashi and Oshima give a sufficient condition on pairs of reductive Lie groups $(G,H)$ which guarantees that $\pi|_H$ contains finite multiplicities in the discrete spectrum of $H$ for each irreducible, unitary representation $\pi$ of $G$. 

These are all very powerful finite multiplicity theorems for branching laws. However, in each case, the finite multiplicity theorem only applies to the discrete part of the direct integral decomposition $\pi|_H$. In this paper, we attempt to study the continuous spectrum of branching laws in a very special case. In order to do this, we make use of many recent ideas of \O rsted and Vargas \cite{OV1},\cite{OV2},\cite{V2}. 

In this paper we will work with reductive Lie groups of Harish-Chandra class. For the sake of completeness, we give the definition of a reductive Lie group of Harish-Chandra class (taken from pages 105-106 of \cite{HC9} where this class of groups was introduced). We say that $G$ is a \emph{reductive Lie group of Harish-Chandra class} if $G$ satisfies the following four conditions:
\begin{itemize}
\item 
$\operatorname{Lie}(G)=\mathfrak{g}$ is a reductive Lie algebra.
\item
$\operatorname{Ad}(G)\subset \operatorname{Int}\mathfrak{g}_{\mathbb{C}}$ where $\mathfrak{g}_{\mathbb{C}}\cong \mathfrak{g}\otimes \mathbb{C}$ is the complexification of $\mathfrak{g}$ and $\operatorname{Int}\mathfrak{g}_{\mathbb{C}}$ is the connected, analytic subgroup of $\operatorname{Aut}\mathfrak{g}_{\mathbb{C}}$ with Lie algebra $\operatorname{Der}\mathfrak{g}_{\mathbb{C}}$, the set of derivations of $\mathfrak{g}_{\mathbb{C}}$.
\item
If $G_1$ is the connected, analytic subgroup of $G$ with Lie algebra $\mathfrak{g}_1=[\mathfrak{g},\mathfrak{g}]$, then the center of $G_1$ is finite.
\item
$[G:G_e]<\infty$ where $G_e$ denotes the identity component of $G$.
\end{itemize}

Here are our results.

\begin{theorem} Suppose $G$ is a reductive Lie group of Harish-Chandra class and suppose $H\subset G^{\tau}\subset G$ is an open subgroup of a symmetric subgroup of $G$. Choose a maximal compact subgroup $K\subset G$ such that $H\cap K\subset H$ is maximal compact, and denote by $\theta$ the associated Cartan involutions of both $G$ and $\mathfrak{g}$. Let $\mathfrak{s}$ be the $-1$ eigenspace of $\theta$ on $\mathfrak{g}$, and let $\mathfrak{q}$ be the  $-1$ eigenspace of $\tau$ on $\mathfrak{g}$. Suppose $H\cong K_0\times H_1$ where $K_0$ is compact and $H_1$ is noncompact. Let $\mathfrak{b}\subset \mathfrak{h}_1\cap \mathfrak{s}$ be a maximal abelian subspace, and let $\mathfrak{a}\subset \mathfrak{s}\cap \mathfrak{q}$ be a maximal abelian subspace. Assume 
$$Z_{H_1\cap K}(\mathfrak{a})Z_{H_1\cap K}(\mathfrak{b})=H_1\cap K.$$
If $\pi$ is a discrete series representation of $G$, then the direct integral decomposition of the restriction $\pi|_H$ has finite multiplicities.
\end{theorem}

Observe that, under slightly weaker hypotheses, it is proved in Theorem 4 of \cite{OV1} that the discrete spectrum of $\pi|_H$ has finite multiplicities. In the last sentence of page 628 in \cite{OV1}, the authors remark that they have not yet determined whether the irreducible continuous factors occur with finite multiplicity in the restriction of a discrete series representation of $\operatorname{Spin}(2n,1)$ to the subgroup $\operatorname{Spin}(2k)\times \operatorname{Spin}(2n-2k,1)$. This generalization of their result answers their question in the affirmative.

One checks that the symmetric pairs $(O(2n,1),O(k)\times O(2n-k,1))$, $(U(n,1),$ $U(k)\times U(n-k,1))$, $(\operatorname{Sp}(n,1),\operatorname{Sp}(k)\times \operatorname{Sp}(n-k,1))$ are all examples of symmetric pairs satisfying the hypotheses of the Theorem. We note that in \cite{KO}, a class of pairs $(G,H)$ was identified such that $\pi|_H$ has finite multiplicities in the discrete spectrum for every irreducible unitary representation $\pi$ of $G$. The above examples are also examples of the class of pairs $(G,H)$ identified in \cite{KO}. However, the class in \cite{KO} contains additional pairs including $(SO(p,q),SO(p-1,q))$ for $p,q>1$ that are not covered by the above Theorem. Most likely these pairs also satisfy a finite multiplicity theorem for the continuous spectrum in discrete series branching laws.

The proof of the Theorem utilizes many techniques from \cite{OV1}. One new technique involves convolution of certain Harish-Chandra Schwartz functions on $G$ with Harish-Chandra characters of irreducible, tempered representations of $H$. This technique is also helpful in proving the following result.

\begin{theorem} Suppose $G$ is a reductive Lie group of Harish-Chandra class, and $H\subset G$ is a closed reductive subgroup of Harish-Chandra class. Let $P=MAN\subset H$ be a cuspidal parabolic subgroup, let $\delta$ be a discrete series representation of $M$, and let $(\widehat{A})'$ be the set of $H$ regular characters of $A$. Let $\sigma(\delta,\nu)$ be the family of irreducible tempered representations corresponding to $\delta\in \widehat{M}$ and $\nu\in (\widehat{A})'$ (see Section 3 for notation). Let $\pi$ be a discrete series representation of $G$, and let $m(\sigma(\delta,\nu),\pi|_H)$ be the multiplicity of the irreducible tempered representation of $\sigma(\delta,\nu)$ in $\pi|_H$. Then $m(\sigma(\delta,\nu),\pi|_H)$ is constant as a function of $\nu\in (\widehat{A})'$. 
\end{theorem}

In \cite{V2}, Vargas considers $m(\sigma(\delta,\nu),\pi|_H)$ as a measurable function of $\nu$ for fixed $\delta$. In Proposition 2, he shows that whenever $m(\sigma(\delta,\nu),\pi|_H)$ is not identically zero almost everywhere as a function of $\nu$, it is zero on a set of measure zero. Thus, our Theorem is a slight generalization of the one obtained in \cite{V2}.

Finally, we attempt to use Kahler geometry to obtain more information about the multiplicities in discrete series branching laws in the special case $(U(n,1),U(1)\times U(n-1,1))$. Let $\psi_m(e^{i\theta})=e^{im\theta}$ denote the unitary character of $U(1)$ with parameter $m\in \mathbb{Z}$. 

\begin{proposition} Suppose $\pi$ is a discrete series representation of $G=U(n,1)$, and consider the symmetric subgroup $H=U(1)\times U(n-1,1)$. Decompose $$\pi|_{U(1)}\cong \bigoplus_{\psi_m\in \widehat{U(1)}} \psi_m\otimes \operatorname{Hom}_{U(1)}(\psi_m,\pi)=\bigoplus_{\psi_m\in \widehat{U(1)}} \psi_m\otimes\tau_m$$
into irreducibles under the action of $U(1)$. Here $\tau_m$ is a unitary representation of $U(n-1,1)$. Next, write
$$\tau_m=\int_{\sigma(\delta,\nu)\in \widehat{U(n-1,1)}_{\text{temp.}}} m(\sigma(\delta,\nu),\tau_m)\sigma(\delta,\nu)$$
as a direct integral of irreducible representations of $U(n-1,1)$ (See Section 3 for notation). Then $m(\sigma(\delta,\nu),\tau_m)\neq 0$ for finitely many elliptic parameters $\delta$ and $m(\sigma(\delta,\nu), \tau_m)<\infty$ for all $\sigma,\nu$. In particular, at most finitely many discrete series occur in $\tau_m$. 
\end{proposition}

One could ask if a similar Proposition holds for all of the symmetric pairs $(O(2n,1),O(k)\times O(2n-k,1))$, $(U(n,1),U(k)\times U(n-k,1))$, and $(\operatorname{Sp}(n,1),\operatorname{Sp}(k)\times \operatorname{Sp}(n-k,1))$. While this may be the case, attempting to prove this generalization using techniques similar to those in the last section of this paper would be computationally unpleasant at best.

We end the introduction with a remark. It seems likely that for most symmetric pairs $(G,H)$, there will exist discrete series $\pi$ such that $\pi|_H$ has only finite multiplicities as well as discrete series $\pi'$ such that $\pi'|_H$ has some infinite multiplicities. For example, if $G$ has holomorphic discrete series, then the holomorphic discrete series of $G$ always restrict with finite multiplicities to any symmetric subgroup (\cite{JV}, \cite{Ko3}, \cite{OlOr}). On the other hand, in many cases one can check using the Mackey restriction principle that large unitary principal series of $G$ will restrict with infinite multiplicities to $H$, and one expects large discrete series to behave the same way. It would be nice to have a criterion to determine when each of these cases occurs that is analogous to the known criteria for discrete decomposability in \cite{Ko3}, \cite{DV}. Certainly, the authors of this paper are very far from understanding this problem.

\section{A Framework for Studying Multiplicities}

Suppose $G$ is a reductive Lie group of Harish-Chandra class and $H\subset G$ is a closed reductive subgroup of Harish-Chandra class. Let $\pi$ be a discrete series representation of $G$, and let $\sigma$ be a tempered representation of $H$. Let $\chi_{\pi}$ be the infinitesimal character of $\pi$, and let $\psi_{\sigma}$ be the infinitesimal character of $\sigma$. Let $K\subset G$ be a maximal compact subgroup such that $H\cap K\subset H$ is a maximal compact subgroup of $H$. Let $\theta$ be the Cartan involution of $\mathfrak{g}=\operatorname{Lie}(G)$ with fixed points $\mathfrak{k}=\operatorname{Lie}(K)$, let $\mathfrak{s}$ be the $-1$ eigenspace of $\theta$, let $\mathfrak{h}=\operatorname{Lie}(H)$ be the Lie algebra of $H$, and let $\mathfrak{q}$ be a $\theta$ stable and $H\cap K$ stable complement to $\mathfrak{h}$ in $\mathfrak{g}$. In particular, we have the decomposition
$$\mathfrak{g}=\mathfrak{k}\cap \mathfrak{h}\oplus \mathfrak{k}\cap \mathfrak{q}\oplus \mathfrak{s}\cap \mathfrak{h}\oplus \mathfrak{s}\cap \mathfrak{q}.$$

Let $(\xi,V_{\xi})$ be the lowest $K$ type of $\pi$,  let $\mathcal{V}_{\xi}$ be the corresponding $G$ equivariant, Hermitian vector bundle on $K\backslash G$, and let 
$$L^2(K\backslash G,\mathcal{V}_{\xi})_{\chi_{\pi}}$$ 
be the $\chi_{\pi}$-eigenspace of the action of $\mathcal{ZU}(\mathfrak{g})$ on space of square integrable sections of $\mathcal{V}_{\xi}$ over $K\backslash G$. Note that the Casimir operator $\Omega\in \mathcal{ZU}(\mathfrak{g})$ acts as an elliptic operator on $K\backslash G$; hence, this is a space of analytic functions by the elliptic regularity theorem. The group $G$ acts on $L^2(K\backslash G,\mathcal{V}_{\xi})_{\chi_{\pi}}$ by right translation, and this representation of $G$ is isomorphic to $\pi$ (\cite{Ho}, \cite{W1}). Let
$$\mathcal{S}(K\backslash G,\mathcal{V}_{\xi})_{\chi_{\pi}}$$ be the dense subspace of those sections that pullback to Harish-Chandra Schwartz functions on $G$ with values in $V_{\xi}$. The Harish-Chandra Schwartz space was first introduced on page 19 of \cite{HC5}.

Whenever $D\in \mathcal{U}(\mathfrak{g})$, left translation by $D$ will be denoted by $L_D$, and right translation by $D$ will be denoted by $R_D$. For instance, if $X\in \mathfrak{g}$, then
$$(L_Xf)(g)=\frac{d}{dt}\Big|_{t=0} f(\exp(-tX)g),\ (R_Xf)(g)=\frac{d}{dt}\Big|_{t=0} f(g\exp(tX)).$$

Recall that if $G$ is a Lie group, then there is a natural bijective linear map 
$$\exp_*: S(\mathfrak{g})\rightarrow \mathcal{U}(\mathfrak{g}).$$
If we regard $S(\mathfrak{g})$ as the space of translation invariant differential operators on $\mathfrak{g}$, and we regard $\mathcal{U}(\mathfrak{g})$ as the space of left invariant differential operators on $G$, then this map is given by 
$$R_{\exp_*(u)}(f)(g):=(u\exp^*(l_{g^{-1}}f))(0).$$
If we instead regard $\mathcal{U}(\mathfrak{g})$ as the space of right invariant differential operators on $G$, then the map can be written
$$L_{\exp_*(u)}(f)(g):=(u\exp^*(\iota^*r_gf))(0).$$
Of course, $l_gf(x)=f(g^{-1}x)$ (respectively $(r_gf)(x)=r(xg)$) denotes left translation (respectively right translation) by $g$, $(\exp^*f)(X)=f(\exp(X))$ denotes the pullback of $f$ by $\exp$, and $(\iota^*f)(x)=f(x^{-1})$ denotes the pullback of $f$ by the inversion map $\iota(x)=x^{-1}$. For more details on this map, see pages 281-283 of \cite{He}.

If $V$ is a real vector space, let $S_m(V)$ denote the elements of the complex symmetric algebra over $V$ with degree at most $m$. The following Proposition is a consequence of Theorem 1.3 in \cite{OV2}, and it is a slight modification of Theorem 3 in \cite{OV1}.

\begin{proposition} [\O rsted-Vargas] There exists a family of continuous, $H$ equivariant maps
$$r_m:\ L^2(K\backslash G,\mathcal{V}_{\xi})_{\chi_{\pi}}\longrightarrow  L^2((H\cap K)\backslash H,\operatorname{Hom}_{\mathbb{C}}(S_m(\mathfrak{s}\cap \mathfrak{q}),\mathcal{V}_{\xi}|_{(H\cap K)\backslash H})).$$
The map $r_m$ is defined by
$$r_m(f)(h)(u)=(L_{\exp_*(u)}f)(h).$$
Moreover, we have 
$$\ker(r_m)\supset \ker(r_{m+1})$$
for all $m$ and 
$$\bigcap_{m=0}^{\infty} \ker(r_m)=\{0\}.$$
\end{proposition}

Observe that our notation differs slightly from the notation in \cite{OV1}. Let
$$r_m=U_m(r_m^*r_m)^{1/2}$$ 
denote the polar decomposition of the operator $r_m$. Then $\{U_m\}$ is a family of $H$ equivariant partial isometries with $\ker r_m=\ker U_m$ and $\operatorname{Image}(U_m)=\overline{\operatorname{Image}(r_m)}$ for every $m$. In particular, $\{\ker(U_m)^{\perp}\}$ is a filtration of $\pi|_H$ via closed, $H$ invariant subspaces. 
We require the following complementary Proposition, which is a close relative of Lemma 1 of \cite{V1}.

\begin{proposition} The map $r_m$ restricts to an $H$ equivariant map between the dense subspaces of Harish-Chandra Schwartz functions
$$r_m: \mathcal{S}(K\backslash G,\mathcal{V}_{\xi})_{\chi_{\pi}}\longrightarrow \mathcal{S}((H\cap K)\backslash H,\operatorname{Hom}_{\mathbb{C}}(S_m(\mathfrak{s}\cap \mathfrak{q}),\mathcal{V}_{\xi}|_{(H\cap K)\backslash H})).$$
\end{proposition}

To prove the proposition, we will need a Lemma. Since $H\cap K\subset H$ is maximally compact, we have compatible Cartan decompositions $\mathfrak{g}=\mathfrak{k}\oplus \mathfrak{s}$ and $\mathfrak{h}=(\mathfrak{h}\cap \mathfrak{k})\oplus (\mathfrak{h}\cap \mathfrak{s})$. Let $\mathfrak{a}_H\subset \mathfrak{h}\cap \mathfrak{s}$ be a maximal abelian subspace, and let $\mathfrak{a}\subset \mathfrak{s}$ be a maximal abelian subspace satisfying $\mathfrak{a}\cap (\mathfrak{h}\cap \mathfrak{s})=\mathfrak{a}_H$. Let $A=\exp(\mathfrak{a})$ and $A_H=\exp(\mathfrak{a}_H)$. Let $\Delta(\mathfrak{g},\mathfrak{a})$ denote the roots of $\mathfrak{g}$ with respect to $\mathfrak{a}$, and choose positive roots $\Delta^+(\mathfrak{g},\mathfrak{a})\subset \Delta(\mathfrak{g},\mathfrak{a})$. Let $\mathfrak{a}^+$ be the positive Weyl chamber for this choice, let $A^+=\exp(\mathfrak{a}^+)$, and let 
$$\rho_{\mathfrak{a}^+}=\frac{1}{2}\sum_{\alpha\in \Delta^+(\mathfrak{g},\mathfrak{a})}\alpha.$$
Let $q:\mathfrak{a}^*\rightarrow \mathfrak{a}_H^*$ be the pullback of the inclusion, and observe 
$$q(\Delta(\mathfrak{g},\mathfrak{a}))\supset \Delta(\mathfrak{h},\mathfrak{a}_H).$$ 
Note that we can choose positive roots $\Delta^+(\mathfrak{h},\mathfrak{a})\subset\Delta(\mathfrak{h},\mathfrak{a})$ so that each positive root $\alpha$ of $\mathfrak{h}$ with respect to $\mathfrak{a}_H$ is of the form $q(\beta)$ for a positive root $\beta\in \Delta^+(\mathfrak{g},\mathfrak{a})$. Fix such a choice and let $\mathfrak{a}_H^+$ be the positive Weyl chamber for these positive roots. Let $A_H^+=\exp(\mathfrak{a}_H^+)$, and let 
$$\rho_{\mathfrak{a}_H^+}=\frac{1}{2}\sum_{\alpha\in \Delta^+(\mathfrak{h},\mathfrak{a}_H)} \alpha.$$

Let $W_G$ be the Weyl group of the root system $\Delta(\mathfrak{g},\mathfrak{a})$, and let $W_H$ be the Weyl group of the root system $\Delta(\mathfrak{h},\mathfrak{a}_H)$. 

\begin{lemma} If $a\in A_H^+$ and $w\in W_G$ such that $wa\in \overline{A^+}$, then 
$$\rho_{\mathfrak{a}^+}(w\log(a))-\rho_{\mathfrak{a}_H^+}(\log(a))\geq 0.$$
\end{lemma}

\begin{proof} Note $\Delta^+(\mathfrak{h},\mathfrak{a}_H)\subset q(w^{-1}\Delta^+(\mathfrak{g},\mathfrak{a}))$. Writing things out, we obtain
$$\rho_{\mathfrak{a}^+}(w\log(a))-\rho_{\mathfrak{a}_H^+}(\log(a))$$
$$=\sum_{\alpha\in \Delta^+(\mathfrak{g},\mathfrak{h})}(w^{-1}\alpha)(\log(a))-\sum_{\alpha\in \Delta^+(\mathfrak{h},\mathfrak{a}_H)}\alpha(\log(a))$$
$$=\sum_{\alpha\in w^{-1}\Delta^+(\mathfrak{g},\mathfrak{a})-\Delta^+(\mathfrak{h},\mathfrak{a}_H)}\alpha(\log(a)).$$
The final sum is nonnegative since $wa\in \overline{A^+}$ implies $\alpha(\log(a))\geq 0$ for all $\alpha\in w^{-1}\Delta^+(\mathfrak{g},\mathfrak{a})$. 
\end{proof}

Next, we prove the proposition.

\begin{proof} To prove this Proposition, it is enough to check that the restriction to $H$ of a Harish-Chandra Schwartz function on $G$ is a Harish-Chandra Schwartz function on $H$. Let $\Xi_G$ be the Harish-Chandra spherical function on $G$, and let $\Xi_H$ be the Harish-Chandra spherical function on $H$ (see for instance page 186 of \cite{Kn} for a definition; in his notation $\Xi_G=\varphi_0^G$). Recall that $f\in C^{\infty}(G)$ lies in the Harish-Chandra Schwartz space, $\mathcal{S}(G)$, of $G$ iff for every $D_1, D_2\in \mathcal{U}(\mathfrak{g})$ and every $m\in \mathbb{N}$, there exists a constant $C_{D_1,D_2,m}>0$ such that
$$\left|(1+|g|)^m(R_{D_1}L_{D_2}f)(g)\right|\leq C_{D_1,D_2,m}\left| \Xi_G(g)\right|\ \text{for\ all}\ g\in G.$$
To define $|\cdot|$, we choose a $W_G$ invariant norm $|\cdot|_{\mathfrak{a}}$ on $\mathfrak{a}$, and we define 
$$|g|=|kak|=|\log(a)|_{\mathfrak{a}}.$$
Of course, the definition of the Harish-Chandra Schwartz space is independent of the choice of norm on $\mathfrak{a}$. Using well-known upper and lower bounds on $\Xi_G$ (see Theorem 3 on page 279 of \cite{HC2} and (8.83) on page 259 of \cite{Kn}), we observe that $f$ is in the Harish-Chandra Schwartz space of $G$ iff for every $D_1,D_2\in \mathcal{U}(\mathfrak{g})$ and $m\in \mathbb{N}$, there exists a constant $C_{D_1,D_2,m}$ such that
$$\left|(R_{D_1}L_{D_2}f)(k_1ak_2)\right|\leq C_{D_1,D_2,m}(1+|\log(a)|_{\mathfrak{a}})^{-m}e^{-\rho_{\mathfrak{a}^+}(\log(a))}$$
for all $k_1,k_2\in K$ and $a\in \overline{A^+}$.

Now, fix $a\in A_H^+$, and choose $w\in W_G$ such that $wa\in \overline{A^+}$. Choose $k'\in K$ that acts on $A$ by $w$, and define $|\cdot|_{\mathfrak{a}_H}$ to be the restriction of $|\cdot|_{\mathfrak{a}}$ to $\mathfrak{a}_{H}$. If $D_1,D_2\in \mathcal{U}(\mathfrak{h})\subset \mathcal{U}(\mathfrak{g})$, then we obtain 
$$\left|(R_{D_1}L_{D_2}f)(k_1ak_2)\right|=\left|(R_{D_1}L_{D_2}f)(k_1k'(k'^{-1}ak')k'^{-1}k_2)\right|$$
$$\leq C_{D_1,D_2,m}(1+|\log(wa)|_{\mathfrak{a}})^{-m}e^{-\rho_{\mathfrak{a}^+}(\log(wa))}$$
$$=C_{D_1,D_2,m}(1+|\log(a)|_{\mathfrak{a}_H})^{-m}e^{-\rho_{\mathfrak{a}^+}(\log(wa))}$$
for $k_1,k_2\in K\cap H$. Applying the Lemma, we obtain
$$\left|(R_{D_1}L_{D_2}f)(k_1ak_2)\right|\leq C_{D_1,D_2,m}(1+|\log(a)|_{\mathfrak{a}_H})^{-m}e^{-\rho_{\mathfrak{a}_H^+}(\log(a))}$$
whenever $D_1,D_2\in \mathcal{ZU}(\mathfrak{h})$, $m\in \mathbb{N}$, $a\in A_H^+$, $k_1,k_2\in K\cap H$, and $f\in \mathcal{S}(G)$. This implies that the restriction of $f$ to $H$ is in $\mathcal{S}(H)$.
\end{proof}

As a corollary of the proof, we observe that the restriction of the map $r_m$ to Harish-Chandra Schwartz spaces is continuous with respect to the locally convex topologies of the Harish-Chandra Schwartz spaces.

Next, we wish to understand the part of the image of the map $r_m$ that transforms by an irreducible tempered representation $\sigma$ of $H$. This will help us understand the multiplicity of the representation $\sigma$ in $\pi|_H$. We need to introduce an additional map.

\bigskip

Suppose $(\eta,V)$ is a finite dimensional, unitary representation of $H\cap K$, and let $\mathcal{V}\rightarrow (H\cap K)\backslash H$ be the corresponding $H$ equivariant, Hermitian vector bundle. Let $L^2((H\cap K)\backslash H,\mathcal{V})$ denote the space of $L^2$ sections of $\mathcal{V}$ over $(H\cap K)\backslash H$ with respect to the $H$ invariant measure on $(H\cap K)\backslash H$, and let $\mathcal{S}((H\cap K)\backslash H,\mathcal{V})\subset L^2((H\cap K)\backslash H,\mathcal{V})$ be the dense subspace of functions that pullback to $V$ valued Harish-Chandra Schwartz functions on $H$.

Let $\Theta_{\sigma}\in C^{-\infty}(H)$ be the character of $\sigma$ on $H$. Recall that in \cite{HC1}, \cite{HC3} it is proved that $\Theta_{\sigma}$ is given by integration against an analytic function on the dense subset of regular, semisimple elements $H'\subset H$ that is locally $L^1$ on all of $H$. We will write $\Theta_{\sigma}$ for the distribution as well as the analytic function on $H'$.

Define $C^{\omega}((H\cap K)\backslash H,\mathcal{V})_{\psi_{\sigma}}$ to be the space of analytic sections of\\
$\mathcal{V}\rightarrow (H\cap K)\backslash H$ 
satisfying
$$R_Df=\psi_{\sigma}(D)f\ \text{for\ all}\ D\in \mathcal{ZU}(\mathfrak{h}).$$

Further, if $\mu\in \widehat{H\cap K}$ is an $H\cap K$ type, let
$$C^{\omega}((H\cap K)\backslash H,\mathcal{V})_{\psi_{\sigma}}(\mu)\subset C^{\omega}((H\cap K)\backslash H,\mathcal{V})_{\psi_{\sigma}}$$
be the subspace that transforms by $\mu$ under right translation by $H\cap K$. 

We define a map
$$*\Theta_{\sigma}: \mathcal{S}((H\cap K)\backslash H,\mathcal{V})\rightarrow C^{\omega}((H\cap K)\backslash H,\mathcal{V})_{\psi_{\sigma}}$$
by $$f\mapsto f*\Theta_{\sigma}\ \text{where}\ (f*\Theta_{\sigma})(h)=\int_H f(hh_1^{-1})\Theta_{\sigma}(h_1)dh_1.$$
The integral converges because integration against $\Theta_{\sigma}$ defines a tempered distribution (see the remarks on page 45 of \cite{HC5} together with Lemma 27 of \cite{HC4} or page 456 of \cite{Kn}).
Note that we may also write
$$(f*\Theta_{\sigma})(h)=\int_H f(h_1)\Theta_{\sigma}(hh_1^{-1})dh_1=\int_H f(h_1^{-1}h)\Theta_{\sigma}(h_1)dh_1.$$
Here we have used substitution and the fact that $\Theta_{\sigma}$ is conjugation invariant. From these formulas, it is clear that $*\Theta_{\sigma}$ is an $H\times H$ equivariant map. In particular, 
$$(f*\Theta_{\sigma})(kh)=\eta(k)^{-1}(f*\Theta_{\sigma})(h)$$
for every $k\in H\cap K$. Moreover, if $\widetilde{D}\in \mathcal{U}(\mathfrak{h})^H$, write $\widetilde{D}=\exp_*(D)$ with $D\in S(\mathfrak{h})^H$, and observe
$$R_{\exp_*D}(f*\Theta_{\sigma})(h)=D_X\int_H f(h\exp(X)h_1^{-1})\Theta_{\sigma}(h_1)dh_1$$
$$=D_X\int_H f(hh_1^{-1}\exp(\operatorname{Ad}_{h_1}X))\Theta_{\sigma}(h_1)dh_1$$
$$=(\operatorname{Ad}_{h_1}D)_X\int_H f(hh_1^{-1}\exp(X))\Theta_{\sigma}(h_1)dh_1$$
$$=D_X\int_H f(hh_1^{-1})\Theta_{\sigma}(h_1\exp(X))dh_1$$
$$=\langle R_{\exp_*D}\Theta_{\sigma},\iota^*l_{h^{-1}}f\rangle=\psi_{\sigma}(\widetilde{D})(f*\Theta_{\sigma})(h).$$
Here we have used that integration against $\Theta_{\sigma}$ is an invariant eigendistribution with infinitesimal character $\psi_{\sigma}$. Finally, the function $f*\Theta_{\sigma}$ is analytic because it is an eigenfunction for the Casimir operator of $H$, which is elliptic on the Riemannian symmetric space $(H\cap K)\backslash H$.

Let $\sigma\in \widehat{H}_{\text{temp.}}$, and let $\mu\in \widehat{H\cap K}$ be a minimal $K$ type for $\sigma$. Let $\phi_{\pi,\sigma,m}$ be the composition of the map $r_m$ and the map $*\Theta_{\sigma}$ from 
$$\mathcal{S}((H\cap K)\backslash H,\operatorname{Hom}_{\mathbb{C}}(S_m(\mathfrak{s}\cap \mathfrak{q}),\mathcal{V}_{\xi}|_{(H\cap K)\backslash H}))$$
to 
$$C^{\omega}((H\cap K)\backslash H,S_m(\mathfrak{s}\cap \mathfrak{q})^*\otimes \mathcal{V}_{\xi})_{\psi_{\sigma}}.$$

Let $m(\sigma,\pi|_H)$ denote the multiplicity of $\sigma$ in the direct integral decomposition of $\pi|_H$.

\begin{lemma} We have the equality 
$$m(\sigma,\pi|_H)=\lim_{m\rightarrow \infty}\dim \phi_{\pi,\sigma,m}(\mathcal{S}(K\backslash G,\mathcal{V}_{\xi})_{\chi_{\pi}}(\mu))$$
$$=\lim_{m\rightarrow \infty}\dim \phi_{\pi,\sigma,m}(\mathcal{S}(K\backslash G,\mathcal{V}_{\xi})_{\chi_{\pi}})(\mu).$$
\end{lemma}

Before beginning the proof of the Lemma, we recall the Plancherel formula for $H$ (\cite{Se}, \cite{M}, \cite{HC7} for the original papers, \cite{HC6}, \cite{F}, \cite{W2} for expository pieces). First, we have an abstract direct integral decomposition
$$L^2(H)\cong \int_{\widehat{H}_{\text{temp.}}} V_{\sigma}^*\otimes V_{\sigma}.$$
Recall that we may view $V_{\sigma}^*\otimes V_{\sigma}\cong \operatorname{End}(V_{\sigma})_{\text{HS}}$ as the space of Hilbert Schmidt operators on $V_{\sigma}$. Let the isomorphism between $L^2(H)$ and the direct integral be denoted $f\mapsto \widehat{f}$. Then for $f\in (L^1\cap L^2)(H)$, we have $\widehat{f}(\sigma)=\sigma(f)\in \operatorname{End}(V_{\sigma})_{\text{HS}}$. In fact, when $f$ is in the Harish-Chandra Schwartz space, we can make things even more explicit. We have a unique $H\times H$ equivariant, continuous map
$$M_{\sigma}:(V_{\sigma}^*\otimes V_{\sigma})^{\infty}\rightarrow C^{\infty}(H)$$
satisfying $$M_{\sigma}(w\otimes v)(h)=\langle \sigma(h)v,w\rangle\ \text{for}\ v\in V_{\sigma}^{\infty}, w\in (V_{\sigma}^*)^{\infty}.$$
Let $C^{\infty}(H)_{\sigma^*\otimes \sigma}$ denote the image of this map. If $\Theta_{\sigma}$ is the Harish-Chandra character of $\sigma$, then we have an $H\times H$ equivariant map $\mathcal{S}(H)\rightarrow C^{\infty}(H)_{\sigma^*\otimes \sigma}$ given by 
$$f\mapsto f*\Theta_{\sigma}.$$
Moreover, for $f\in \mathcal{S}(H)$, we have $$M_{\sigma}(\sigma(f))=f*\Theta_{\sigma}.$$

Now, let $V$ be a finite dimensional complex vector space, and let $W\subset L^2(H)\otimes V$ be a closed, right $H$ invariant subspace such that $\mathcal{S}(W):=W\cap (\mathcal{S}(H)\otimes V)\subset W$ is dense. Then for each $\sigma\in \widehat{H}_{\text{temp.}}$, we must have $W_{\sigma}\subset V_{\sigma}^*\otimes V$ such that $$W\cong \int_{\widehat{H}_{\text{temp.}}} W_{\sigma}\otimes V_{\sigma}.$$ 
The point is that we can use our explicit Fourier transform to describe $W_{\sigma}$. Extend $M_{\sigma}$ to a map on $\mathcal{S}(H)\otimes V$ (trivially on $V$), and note that $$\{f*\Theta_{\sigma}|\ f\in \mathcal{S}(W)\}\subset M_{\sigma}(V_{\sigma}\otimes W_{\sigma})^{\infty}$$
is a dense subspace. Since the maps $*\Theta_{\sigma}$ and $M_{\sigma}$ are $H\times H$ equivariant, we may look at the part of both sides that transforms on the right by the minimal $K$ type $\mu$ of $\sigma$. Then we obtain
$$\{f*\Theta_{\sigma}|\ f\in \mathcal{S}(W)(\mu)\}\subset M_{\sigma}(V_{\sigma}(\mu)\otimes W_{\sigma})^{\infty}$$
is dense. Since $\mu$ is a minimal $H\cap K$ type of $\sigma$, it follows from \cite{Vo} that $\dim V_{\sigma}(\mu)=1$. In particular, 
$$\dim W_{\sigma}=\dim \{f*\Theta_{\sigma}|\ f\in \mathcal{S}(W)(\mu)\}.$$
Finally, note that $\dim W_{\sigma}=m(\sigma, W)$, the multiplicity of $\sigma$ in $W$.

\begin{proof} Now, to compute the multiplicity of $\sigma$ in $\pi|_H$, we first realize $\pi$ as\\ $L^2(K\backslash G,\mathcal{V}_{\xi})_{\chi_{\pi}}$. Recall $\{\ker(U_m)^{\perp}\}$ is a filtration of $\pi|_H$ by $H$ invariant, closed subspaces. Since each $\{\ker(U_m)^{\perp}\}$ can be identified with a closed subpace of $$L^2((H\cap K)\backslash H,\operatorname{Hom}_{\mathbb{C}}(S_m(\mathfrak{s}\cap \mathfrak{q}),\mathcal{V}_{\xi}|_{(H\cap K)\backslash H}))\subset L^2(H)\otimes \operatorname{Hom}_{\mathbb{C}}(S_m(\mathfrak{s}\cap \mathfrak{q}),V_{\xi}),$$ we may take the measure on $\widehat{H}$ in the direct integral decomposition of $\ker(U_m)^{\perp}$ to be the Plancherel measure on $\widehat{H}$ for every $m$. In particular, this implies 
$$m(\sigma,\pi|_H)=\lim_{m\rightarrow \infty} m(\sigma, \ker(U_m)^{\perp}).$$

We may identify $\ker(U_m)^{\perp}$ as an $H$ representation with 
$$\overline{r_m(L^2(K\backslash G, V_{\xi})_{\chi_{\pi}})}\subset L^2((H\cap K)\backslash H, S_m(\mathfrak{s}\cap \mathfrak{q})^*\otimes V_{\xi}).$$
Since Harish-Chandra Schwartz functions on $G$ are dense in $L^2(K\backslash G, V_{\xi})_{\chi_{\pi}}$ and the restriction of Harish-Chandra Schwartz functions to $H$ are Harish-Chandra Schwartz functions on $H$ by Proposition 2.2, we may use the above discussion of the Plancherel formula to compute the multiplicity of $\sigma$ in $\ker(U_m)^{\perp}$. Explicitly, we have 
$$m(\sigma, \ker(U_m)^{\perp})=\dim \phi_{\pi,\sigma, m}(\mathcal{S}(K\backslash G, \mathcal{V}_{\xi})_{\chi_{\pi}}(\mu)).$$
The lemma follows.
\end{proof}

\section{A Constant Multiplicity Theorem}

Given a cuspidal parabolic subgroup $P=MAN\subset H$, a discrete series representation $\delta\in \widehat{M}$ and a unitary character $\nu\in \widehat{A}$, we may form the (possibly infinite dimensional) vector bundle on $G/P$ corresponding to the tensor product of $\delta\otimes \nu\otimes 1$ with the square root of the density bundle on $G/P$. The space of $L^2$ sections of this vector bundle is a tempered representation of $H$, which we will call $\sigma(\delta,\nu)$; if $\nu$ is regular it is irreducible \cite{HC7}. Moreover, every irreducible tempered representation $\sigma\in \widehat{H}$ with regular infinitesimal character is of this form with $\nu\in \widehat{A}$ regular \cite{T}.

Let $(\widehat{A})'$ denote the set of regular characters of $A$, and consider $\{\sigma(\delta,\nu)|\ \nu\in (\widehat{A})'\}$ for fixed $\delta\in \widehat{M}$. This set of irreducible tempered representations is an open subset of $\widehat{H}_{\text{temp.}}$ and the union of all such open subsets has full measure in $\widehat{H}_{\text{temp.}}$ \cite{Fe}, \cite{HC6}, \cite{HC7}, \cite{Kn}. In particular, we may study the multiplicity function $m(\sigma(\delta,\nu),\pi)$ as a function of $\nu\in (\widehat{A})'$ for fixed $\pi\in \widehat{G}_{\text{disc.}}$ and $\delta\in \widehat{M}_{\text{disc.}}$. 

\begin{theorem} Let $\pi\in \widehat{G}_{\text{disc.}}$ be a discrete series representation of $G$, a reductive Lie group of Harish-Chandra class, let $H\subset G$ be a closed reductive subgroup of $G$ of Harish-Chandra class, and let $P=MAN\subset H$ be a cuspidal parabolic subgroup of $H$. Let $\delta\in \widehat{M}_{\text{disc.}}$ be a discrete series representation of $M$, and let $\nu\in (\widehat{A})'$ be a regular, unitary character of $A$. Let $\sigma(\delta,\nu)$ be the irreducible tempered representation of $H$ induced from the representation $\delta\otimes \nu\otimes 1$ of $P=MAN$. Let  
$$m(\sigma(\delta,\nu),\pi|_H)$$
be the multiplicity function that records the multiplicity of $\sigma(\delta,\nu)$ in the direct integral decomposition of $\pi|_H$. Then, as a function of $\nu\in (\widehat{A})'$, the multiplicity function $m(\sigma(\delta,\nu),\pi|_H)$ is constant.
\end{theorem}

In \cite{V2}, Vargas considers $m(\sigma(\delta,\nu),\pi|_H)$ as a measurable function of $\nu$ for fixed $\delta$. In Proposition 2, he shows that whenever $m(\sigma(\delta,\nu),\pi|_H)$ is not identically zero almost everywhere as a function of $\nu$, it is zero on a set of measure zero. In particular, this Theorem is a generalization of Proposition 2 of \cite{V2}.

\begin{proof} By Lemma 2.4, we have the formula
$$m(\sigma(\delta,\nu),\pi|_H)=\lim_{m\rightarrow \infty} \dim \phi_{\pi,\sigma,m}(\mathcal{S}(K\backslash G,\mathcal{V}_{\xi})_{\chi_{\pi}}(\mu)).$$
Fix $\delta$ and $\pi$, and first assume that $m(\sigma(\delta,\nu),\pi|_H)$ is bounded almost everywhere for $\nu\in (\widehat{A})'$ with maximum $N$. Choose $\nu_0\in (\widehat{A})'$ such that $N=m(\sigma(\delta,\nu_0),\pi|_H)$, and choose $f_1,\ldots,f_N\in \mathcal{S}(K\backslash G,\mathcal{V}_{\xi})_{\chi_{\pi}}(\mu)$ such that 
$$\left\{\phi_{\pi,\sigma(\delta,\nu_0),m}(f_j)=r_m(f_j)*\Theta_{\sigma(\delta,\nu_0)}\right\}$$
is linearly independent for sufficiently large $m$. Let $m_0$ be a natural number that is sufficiently large. Recall $$r_m(f_j)*\Theta_{\sigma(\delta,\nu_0)}\in C^{\omega}((H\cap K)\backslash H,S_m(\mathfrak{s}\cap \mathfrak{q})^*\otimes \mathcal{V}_{\xi})_{\psi_{\sigma}}(\mu).$$
In particular, one may think of each $r_m(f_j)*\Theta_{\sigma(\delta,\nu_0)}$ as a complex valued function on $H\times S_m(\mathfrak{s}\cap \mathfrak{q})\otimes V_{\xi}^*$. We require the following Lemma.

\begin{lemma} Let $X$ be a set, and let $\{F_1,\ldots,F_N\}$ be a linearly independent set of complex valued functions on $X$. Then there exists a subset $\{x_1,\ldots,x_N\}\subset X$ such that 
$$\det(F_i(x_j))\neq 0.$$
\end{lemma}

The proof of this Lemma is an easy linear algebra exercise that is left to the reader. In our case, we may find $(h_j,w_j)\in H\times S_{m_0}(\mathfrak{s}\cap \mathfrak{q})\otimes V_{\xi}^*$ for $j=1,\ldots,N$ such that 
$$\det((r_{m_0}(f_i)*\Theta_{\sigma(\delta,\nu_0)})(h_j,w_j))\neq 0.$$
Now, observe that 
$$\nu\mapsto \det((r_{m_0}(f_i)*\Theta_{\sigma(\delta,\nu)})(h_j,w_j))$$
is an analytic function on $\widehat{A}$ since $\Theta_{\sigma(\delta,\nu)}$ is analytic as a function of $\nu$ (Observe that $\Theta_{\sigma(\delta,\nu)}$ is well-defined and analytic function for all $\nu\in \widehat{A}$ even if $\sigma(\delta,\nu)$ isn't irreducible for singular $\nu$). In particular, since this function is non-zero somewhere, it is non-zero almost everywhere. Therefore,
$$\{r_{m_0}(f_j)*\Theta_{\sigma(\delta,\nu)}\}$$
is a linearly independent set for almost all $\nu\in (\widehat{A})'$. Since $\mu$ is a lowest $K$ type of $\sigma(\delta,\nu_0)$, it is also a lowest $K$ type for $\sigma(\delta,\nu)$ for every $\nu\in (\widehat{A})'$ (an easy calculation shows that all of these representations have the same $K$ types). In particular, Lemma 2.4 implies that 
$$m(\sigma(\delta,\nu),\pi|_H)\geq N$$
for almost every $\nu\in (\widehat{A})'$. Since $N$ was chosen to be the maximum, $m(\sigma(\delta,\nu))$ is the constant function $N$ almost everywhere on the subset $\{\sigma(\delta,\nu)|\ \nu\in (\widehat{A})'\}$. Because $m$ is a measurable function that is only well-defined up to subsets of measure zero anyway, we might as well say $m$ is the constant function $N$ on $\{\sigma(\delta,\nu)|\ \nu\in (\widehat{A})'\}$.

Next, suppose $\{m(\sigma(\delta,\nu),\pi|_H)|\ \nu\in (\widehat{A})'\}$ is unbounded. Thus, for any $N$, we may find $\nu_0$ such that $m(\sigma(\delta,\nu_0),\pi|_H)\geq N$, and the above argument then implies $m(\sigma(\delta,\nu),\pi|_H)\geq N$ for almost every $\nu\in (\widehat{A})'$. Since this is true for every $N$, we observe $m(\sigma(\delta,\nu),\pi|_H)=\infty$ for almost every $\nu\in (\widehat{A})'$. Again, since $m$ is a measurable function, we might as well call it the constant function $\infty$ on all of $\{\sigma(\delta,\nu)|\ \nu\in (\widehat{A})'\}$.
\end{proof}

\section{A Space of Spherical Functions}

Now, we specialize to the case where $H\subset G^{\tau}\subset G$ is an open subgroup of a symmetric subgroup of $G$. Let $\chi_{G}$ be a character of the center of the universal enveloping algebra $\mathcal{ZU}(\mathfrak{g})$, and let $\psi_{H}$ be a character of $\mathcal{ZU}(\mathfrak{h})$. Let $(\xi,V_{\xi})$ be an irreducible representation of $K$, and let $\mu\in \widehat{K\cap H}$. Define $$C^{\omega}(K\backslash G,\mathcal{V}_{\xi})_{\chi_{G},\psi_H}(\mu)$$ to be the vector space of analytic sections of the bundle $\mathcal{V}_{\xi}\rightarrow K\backslash G$ satisfying

\begin{enumerate}
\item
$$R_Df=\chi_{G}(D)f\ \text{for\ all}\ D\in \mathcal{ZU}(\mathfrak{g}).$$
\item
$$R_Df=\psi_{H}(D)f\ \text{for\ all}\ D\in \mathcal{ZU}(\mathfrak{h}).$$
\item 
If we let $K\cap H$ act on the space of functions on $G$, $$\operatorname{Span}_{\mathbb{C}}\left\{f(gk')|\ k'\in K\cap H\right\}$$ by right translation, then the corresponding representation is isomorphic to a Hilbert space sum of copies of $\mu$.
\end{enumerate}

Recall that we have the decomposition
$$\mathfrak{g}=\mathfrak{k}\cap \mathfrak{h}\oplus \mathfrak{k}\cap \mathfrak{q}\oplus \mathfrak{s}\cap \mathfrak{h}\oplus \mathfrak{s}\cap \mathfrak{q}$$
Here $\mathfrak{s}$ is the $-1$ eigenspace of the Cartan involution of $\mathfrak{g}$ with respect to $\mathfrak{k}$, and $\mathfrak{q}$ is the $-1$ eigenspace of the action of $\tau$ on $\mathfrak{g}$ (we write $\tau$ for both the involution on $G$ and its differential on $\mathfrak{g}$).
The following Theorem is a slight generalization of Lemma 1 on page 623 of \cite{OV1}.

\begin{theorem} [\O rsted-Vargas] Suppose $H\subset G^{\tau} \subset G$ is an open subgroup of a symmetric subgroup of a reductive Lie group of Harish-Chandra class $G$, let $\chi_G$ be a character of $\mathcal{ZU}(\mathfrak{g})$, and let $\psi_H$ be a character of $\mathcal{ZU}(\mathfrak{h})$. Further, suppose $H\cong K_0\times H_1$ where $K_0$ is compact and $H_1$ is noncompact. Let $\mathfrak{b}\subset \mathfrak{h}_1\cap \mathfrak{s}$ be a maximal abelian subspace, and let $\mathfrak{a}\subset \mathfrak{s}\cap \mathfrak{q}$ be a maximal abelian subspace. Suppose 
$$Z_{H_1\cap K}(\mathfrak{b})Z_{H_1\cap K}(\mathfrak{a})=H_1\cap K.$$
Then 
$$C^{\omega}(K\backslash G,\mathcal{V}_{\xi})_{\chi_{G},\psi_H}(\mu)$$
is finite dimensional for every $\mu\in \widehat{K\cap H}$.

\end{theorem}

One observes that the symmetric pairs $(O(2n,1),O(k)\times O(2n-k,1))$, $(U(n,1),$\\ $U(k)\times U(n-k,1))$, and $(\operatorname{Sp}(n,1),\operatorname{Sp}(k)\times \operatorname{Sp}(n-k,1))$ are all examples of symmetric pairs that satisfy the hypothesis of the above Theorem.

Technically, this Theorem is a very slight generalization of Lemma 1 of \cite{OV1}. Thus, we point out how to prove this version. The key is a Lemma of Harish-Chandra.

\begin{lemma} [Harish-Chandra] Suppose $H_1$ is a reductive Lie group of Harish-Chandra class with maximal compact subgroup $K_1\subset H_1$, and let $\mathfrak{h}_1=\mathfrak{k}_1+\mathfrak{s}_1$ be the corresponding Cartan decomposition of $\mathfrak{h}_1$. Let $\mathfrak{b}\subset \mathfrak{s}_1$ be a maximal abelian subspace, and let $M=Z_{K_1}(\mathfrak{b})$. Suppose $K_2\subset K_1$ is a subgroup satisfying $K_2M=K_1$. Let $(\mu_1,V_1)\in \widehat{K_1}$, let $(\mu_2,V_2)$ be a finite dimensional representation of $K_2$, and let $\psi$ be a character of $\mathcal{ZU}(\mathfrak{h}_1)$. Let 
$$C^{\omega}(K_2\backslash H_1,\mathcal{V}_2)_{\psi}(\mu_1)$$
be the space of analytic sections of the vector bundle $\mathcal{V}_2\rightarrow K_2\backslash H$ that transform by $\mu_1$ on the right and by $\psi$ under $\mathcal{ZU}(\mathfrak{h})$. Then
$$C^{\omega}(K_2\backslash H_1,\mathcal{V}_2)_{\psi}(\mu_1)$$
is finite dimensional.
\end{lemma}

This Lemma is a very slight modification and generalization of Corollary 2 on page 65 of \cite{HC8} or Lemma 8 on page 67 of \cite{HC8}. For the sake of completeness, we give a proof.

\begin{proof} Choose $b\in \exp(\mathfrak{b})$ regular. Then Corollary 1 of \cite{HC8} says that there exists a finite collection $\{v_i\}_{i=1}^r\subset \mathcal{U}(\mathfrak{b})$ such that
$$\mathcal{U}(\mathfrak{h}_1)=\left(\operatorname{Ad}(b^{-1})\mathcal{U}(\mathfrak{k}_2)\right)\left(\sum_{i=1}^r\mathcal{ZU}(\mathfrak{h}_1)v_i\right)\mathcal{U}(\mathfrak{k}_1).$$ 
Now, view $C^{\omega}(K_2\backslash H_1,\mathcal{V}_2)_{\psi}(\mu_1)$ as a space of analytic functions from $H_1$ to $V_2$. Denote the annihilator of $V_1$ in $\mathcal{U}(\mathfrak{k}_1)$ by $\operatorname{Ann}(V_1)$, and let $\{w_j\}_{j=1}^{m}$ be a basis for the finite dimensional vector space $\mathcal{U}(\mathfrak{k}_1)/\operatorname{Ann}(V_1)$. We have a map
$$C^{\omega}(K_2\backslash H_1,\mathcal{V}_2)_{\psi}(\mu_1)\rightarrow V_2\otimes \mathbb{C}^{r+m}$$
by $$f\mapsto (R_{v_i}R_{w_j}f)(b).$$
If we show that this map is injective, then the Lemma follows. Therefore, we suppose $f\in C^{\omega}(K_2\backslash H_1,\mathcal{V}_2)_{\psi}(\mu_1)$ with $(R_{v_i}R_{w_j}f)(b)=0$ for all $i,j$. Since $f$ is analytic, to conclude $f=0$, it is enough to show $(R_Df)(b)=0$ for all $D\in \mathcal{U}(\mathfrak{h}_1)$. Using the above decomposition of $\mathcal{U}(\mathfrak{h}_1)$, we write $D=(\operatorname{Ad}(b^{-1})D_2)(\sum z_iv_i)D_1$ with $D_2\in \mathcal{U}(\mathfrak{k}_2)$, $z_i\in \mathcal{ZU}(\mathfrak{h}_1)$ and $D_2\in \mathcal{U}(\mathfrak{k}_1)$. Further, write $D_1=\sum_{j=1}^m a_j w_j \operatorname{mod}(\operatorname{Ann}(V_1))$. Then
$$(R_Df)(b)=d\mu_2(D_1)\left(\sum_{i,j} \psi(z_i)a_j R_{v_i}R_{w_j}f\right)(b)=0.$$
The Lemma follows.
\end{proof}

Now, we explain how to modify the proof on pages 623-625 of \cite{OV1} to obtain Theorem 3.1 under our slightly weaker assumptions. Let $\mathfrak{a}\subset \mathfrak{s}\cap \mathfrak{q}$ be a maximal abelian subspace, and let $a\in \exp(\mathfrak{a})$ be regular. Given $f\in C^{\omega}(K\backslash G,\mathcal{V}_{\xi})_{\chi_{G},\psi_H}(\mu)$, \O rsted and Vargas define the middle differentiation of $f$ by $X\in \mathfrak{g}$ at $(a,h)$ to be
$$f(a;X;h)=\frac{d}{dt}\Big|_{t=0} f(a\exp(tX)h).$$
They extend this definition to any $X\in \mathcal{U}(\mathfrak{g})$ by iteration. Then, utilizing arguments on page 624, \O rsted and Vargas show that there is a finite set $\{v_i\}\subset \mathcal{ZU}(Z_{\mathfrak{g}}(\mathfrak{a}))$ such that $f\in C^{\omega}(K\backslash G,\mathcal{V}_{\xi})_{\chi_{G},\psi_H}(\mu)$ is determined by the set of functions $H_1\rightarrow V_1$ of the form
$$h_1\mapsto g_{f,k_0}(h_1)=f(a;v_j;k_0h_1)$$
for all $k_0\in K_0$ and $j$. This part of their argument is unchanged by our weaker set of hypotheses. Let $\psi_{H_1}$ be the restriction of $\psi_H$ to $\mathcal{ZU}(\mathfrak{h}_1)$, and let $\zeta_i$ for $i=1,\ldots,r$ be the finite collection of irreducible $H_1\cap K=K_1$ representations in $\mu|_{H_1\cap K}$. Now, it is clear that 
$$g_{f,k_0}\in \operatorname{Span}_{i=1}^r C^{\omega}(K_2\backslash H_1,\mathcal{V}_2)_{\psi_{H_1}}(\zeta_i)\subset C^{\omega}(K_2\backslash H_1,\mathcal{V}_2)_{\psi_{H_1}}.$$
In the last Lemma, we showed that this space is finite dimensional. Following \cite{OV1}, we then consider $g_{f,k_0}(h_1)$ as a function on $K_0$ with values in the finite dimensional space $\operatorname{Span}_{i=1}^r C^{\omega}(K_2\backslash H_1,\mathcal{V}_2)_{\psi_{H_1}}(\zeta_i)$. Note that $\mathcal{ZU}(\mathfrak{k}_0)$ acts by a fixed character on $k_0\mapsto g_{f,k_0}$. Moreover, the space of vector valued functions on a compact group $K_0$ in a fixed eigenspace for $\mathcal{ZU}(\mathfrak{k}_0)$ is always a finite dimensional space. The Theorem follows.

\section{A Finite Multiplicity Theorem}

We return to the assumptions of the first section. In particular, $G$ is a reductive Lie group of Harish-Chandra class, $H\subset G$ is a closed reductive subgroup of Harish-Chandra class, and $K\subset G$ is a maximal compact subgroup, chosen in such a way that $H\cap K\subset H$ is maximal compact. We wish to use the space $C^{\omega}(K\backslash G,\mathcal{V}_{\xi})_{\chi_{G},\psi_{H}}(\mu)$ introduced in the last section to study multiplicities. To do this, we need to add two additional maps to the framework developed in section one.

Define
$$*_H\Theta_{\sigma}: \mathcal{S}(K\backslash G,\mathcal{V}_{\xi})_{\chi_{\pi}}\rightarrow C^{\omega}(K\backslash G,\mathcal{V}_{\xi})_{\chi_{\pi},\psi_{\sigma}}$$
by $$f\mapsto f*_H\Theta_{\sigma}\ \text{where}\ (f*_H\Theta_{\sigma})(g)=\int_H f(gh^{-1})\Theta_{\sigma}(h)dh.$$
This definition depends on a choice of Haar measure on $H$; thus, our map is well-defined only up to a positive constant. However, this constant is inconsequential for our applications; thus, we will ignore this ambiguity.

The integral converges because $l_{g^{-1}}\iota^*f$ is in the Harish-Chandra Schwartz space of $G$ for every $g\in G$ (see Lemma 13 of \cite{HC5}); thus, by the proof of Proposition 2.2, its restriction to $H$ is in the Harish-Chandra Schwartz space of $H$. Convergence now follows from the fact that integration against $\Theta_{\sigma}$ defines a tempered distribution (see the remarks on page 45 of \cite{HC5} together with Lemma 27 of \cite{HC4} or page 456 of \cite{Kn}).

We must check that the image of $*_H\Theta_{\sigma}$ is indeed contained in $C^{\omega}(K\backslash G,\mathcal{V}_{\xi})_{\chi_{\pi},\psi_{\sigma}}$. Note that $f\mapsto f*_H\Theta_{\sigma}$ is left $G$ invariant as a map from functions on $G$ to functions on $G$; in particular, $f*_H\Theta_{\sigma}$ transforms by $\xi$ on the left and can be thought of as a section of the appropriate vector bundle. To check the $\chi_{\pi}$ condition, observe that the map $\exp_*$ defined earlier is equivariant for the adjoint action. In particular, every $\widetilde{D}\in \mathcal{ZU}(\mathfrak{g})=\mathcal{U}(\mathfrak{g})^G$ is of the form $\exp_*(D)$ for $D\in S(\mathfrak{g})^G$. Choosing such a $D$, we observe
$$R_{\exp_*(D)}(f*\Theta_{\sigma})(g)=D_X\int_H f(g\exp(X)h^{-1})\Theta_{\sigma}(h)dh$$
$$=D_X\int_H f(gh^{-1}\exp(\operatorname{Ad}_hX))\Theta_{\sigma}(h)dh$$
$$=(\operatorname{Ad}_h^{-1}D)_X\int_Hf(gh^{-1}\exp(X))\Theta_{\sigma}(h)dh$$
$$=\int_H D_Xf(gh^{-1}\exp(X))\Theta_{\sigma}(h)dh=\chi_{\pi}(\widetilde{D})(f*\Theta_{\sigma})(g).$$

To check the $\psi_{\sigma}$ condition, let $D\in S(\mathfrak{h})^H$ be homogeneous and observe
$$R_{\exp_*(D)}(f*\Theta_{\sigma})(g)=D_X\int_H f(g\exp(X)h^{-1})\Theta_{\sigma}(h)dh$$
$$=D_X\int_H f(g(h\exp(-X))^{-1})\Theta_{\sigma}(h)dh=\langle \Theta_{\sigma}, (-1)^{\deg D}R_{\exp_*(D)}l_{g^{-1}}\iota^*f\rangle$$
$$\langle R_{\exp_*(D)}\Theta_{\sigma}, l_{g^{-1}}\iota^*f\rangle=\psi_{\sigma}(\exp_*D)\langle \Theta_{\sigma},l_{g^{-1}}\iota^*f\rangle$$
$$=\psi_{\sigma}(\exp_*D)(f*\Theta_{\sigma})(g).$$
Here we have used that $\Theta_{\sigma}$ is an invariant eigendistribution with infitesimal character $\psi_{\sigma}$. 
Note that $f*_H\Theta_{\sigma}$ is analytic by the elliptic regularity theorem, since the Casimir operator of $G$ acts on $K\backslash G$ by an elliptic operator.

\bigskip

To complete our commutative diagram, we need one final map. Define
$$\widetilde{r}_m: C^{\omega}(K\backslash G,\mathcal{V}_{\xi})_{\chi_{\pi},\psi_{\sigma}}\longrightarrow C^{\omega}((H\cap K)\backslash H,S_m(\mathfrak{s}\cap \mathfrak{q})^*\otimes \mathcal{V}_{\xi})_{\psi_{\sigma}}$$
by 
$$\widetilde{r}_m(f)(h)(D)=(L_{\exp_*(D)}f)(h)$$
for $D\in S_m(\mathfrak{s}\cap \mathfrak{q})$. It is not difficult to verify that this map is right $H$ equivariant and left $H\cap K$ equivariant. Hence, its image does indeed lie in 
$$C^{\omega}((H\cap K)\backslash H,S_m(\mathfrak{s}\cap \mathfrak{q})^*\otimes \mathcal{V}_{\xi})_{\psi_{\sigma}}.$$

\bigskip

Here is the commutative diagram formed from the four maps we have introduced.

$$\begin{CD}
\mathcal{S}(K\backslash G,\mathcal{V}_{\xi})_{\chi_{\pi}}       @> r_m >>     \mathcal{S}((H\cap K)\backslash H,\operatorname{Hom}_{\mathbb{C}}(S_m(\mathfrak{s}\cap \mathfrak{q}),\mathcal{V}_{\xi}))\\
@VV *_H\Theta_{\sigma} V                                                @VV *\Theta_{\sigma} V\\
C^{\omega}(K\backslash G,\mathcal{V}_{\xi})_{\chi_{\pi},\psi_{\sigma}}           @> \widetilde{r}_m >>  C^{\omega}((H\cap K)\backslash H,S_m(\mathfrak{s}\cap \mathfrak{q})^*\otimes \mathcal{V}_{\xi})_{\psi_{\sigma}}
\end{CD}$$

Next, we check that this diagram commutes. Observe
$$\widetilde{r}_m(f*\Theta_{\sigma})(D,h)=L_{\exp_*(D)}(f*_H\Theta_{\sigma})(h)$$
$$=L_{\exp_*(D)}\int_H f(hh_1^{-1})\Theta_{\sigma}(h_1)dh_1$$
$$=D_X\int_H f(\exp(-X)hh_1^{-1})\Theta_{\sigma}(h_1)dh_1$$
$$=\int_H (D_Xf)(\exp(-X)hh_1^{-1})\Theta_{\sigma}(h_1)dh_1$$
$$=\int_H (L_{\exp_*(D)})(hh_1^{-1})\Theta_{\sigma}(h_1)dh_1$$
$$=(r_mf*\Theta_{\sigma})(h,D).$$

Now, since all four of the maps, $r_m$, $\widetilde{r}_m$, $*_H\Theta_{\sigma}$, and $*\Theta_{\sigma}$ are right $H$ equivariant, they all preserve $H\cap K$ types. Hence, we may take the piece of each of our four spaces that transforms on the right by a minimal $H\cap K$ type $\mu$ of $\sigma$, and we obtain a commutative diagram.

$$\begin{CD}
\mathcal{S}(K\backslash G,\mathcal{V}_{\xi})_{\chi_{\pi}}(\mu)       @> r_m >>     \mathcal{S}((H\cap K)\backslash H,\operatorname{Hom}_{\mathbb{C}}(S_m(\mathfrak{s}\cap \mathfrak{q}),\mathcal{V}_{\xi}))(\mu)\\
@VV *_H\Theta_{\sigma} V                                                @VV *\Theta_{\sigma} V\\
C^{\omega}(K\backslash G,\mathcal{V}_{\xi})_{\chi_{\pi},\psi_{\sigma}}(\mu)           @> \widetilde{r}_m >>   C^{\omega}((H\cap K)\backslash H,S_m(\mathfrak{s}\cap \mathfrak{q})^*\otimes \mathcal{V}_{\xi})_{\psi_{\sigma}}(\mu)
\end{CD}$$

Using this commutative diagram together with Lemma 2.4, we obtain the following.

\begin{proposition} Suppose $\pi$ is a discrete series representation of a reductive Lie group of Harish-Chandra class $G$ with lowest $K$ type $(\xi,V_{\xi})$ and infinitesimal character $\chi_{\pi}$, and suppose $\sigma$ is a tempered representation of a closed reductive subgroup $H\subset G$ of Harish-Chandra class with minimal $H\cap K$ type $\mu$ and infinitesimal character $\psi_{\sigma}$. If $m(\sigma,\pi|_H)$ denotes the multiplicity of $\sigma$ in the direct integral decomposition of $\pi|_H$, then
$$m(\sigma,\pi|_H)\leq \dim C^{\omega}(K\backslash G,\mathcal{V}_{\xi})_{\chi_{\pi},\psi_{\sigma}}(\mu).$$
In particular, if $C^{\omega}(K\backslash G,\mathcal{V}_{\xi})_{\chi_{\pi},\psi_{\sigma}}(\mu)$ is finite dimensional, then the multiplicity of $\sigma$ in $\pi|_H$ is finite.
\end{proposition}

\begin{proof} To prove the proposition, we utilize Lemma 2.4 and the above commutative diagram. We may write $\phi_{\pi,\sigma,m}=\widetilde{r}_m\circ *_H {\Theta_{\sigma}}$, and the dimension of the image of $\widetilde{r}_m$ cannot exceed the dimension of its domain. Taking the limit and applying Lemma 2.4, the result follows.
\end{proof}

\begin{theorem} Suppose $G$ is a reductive Lie group of Harish-Chandra class and suppose $H\subset G^{\tau}\subset G$ is an open subgroup of a symmetric subgroup of $G$. Choose a maximal compact subgroup $K\subset G$ such that $H\cap K\subset H$ is maximal compact, and denote by $\theta$ the associated Cartan involutions of both $G$ and $\mathfrak{g}$. Let $\mathfrak{s}$ be the $-1$ eigenspace of $\theta$ on $\mathfrak{g}$, and let $\mathfrak{q}$ be the  $-1$ eigenspace of $\tau$ on $\mathfrak{g}$. Suppose $H\cong K_0\times H_1$ where $K_0$ is compact and $H_1$ is noncompact. Let $\mathfrak{b}\subset \mathfrak{h}_1\cap \mathfrak{s}$ be a maximal abelian subspace, and let $\mathfrak{a}\subset \mathfrak{s}\cap \mathfrak{q}$ be a maximal abelian subspace. Assume 
$$Z_{H_1\cap K}(\mathfrak{a})Z_{H_1\cap K}(\mathfrak{b})=H_1\cap K.$$
If $\pi$ is a discrete series representation of $G$, then the direct integral decomposition of the restriction $\pi|_H$ has finite multiplicities.
\end{theorem}

Observe that $H$ may not be of Harish-Chandra class even though we have assumed $G$ to be of Harish-Chandra class. However, we observe that the identity component $H_e$ is always of Harish-Chandra class. Indeed, referring to the definition of Harish-Chandra class in the introduction, the first condition follows from the well-known fact that the fixed point set of a complex reductive Lie algebra under an involutive automorphism is a complex reductive Lie algebra. For the second condition, we observe 
$$\operatorname{Ad}(G^{\tau})\subset \operatorname{Ad}(G)^{\tau}\subset (\operatorname{Int}\mathfrak{g}_{\mathbb{C}})^{\tau}.$$
Therefore, $$\operatorname{Ad}(H_e)=\operatorname{Ad}((G^{\tau})_e)\subset (\operatorname{Int}(\mathfrak{g}_{\mathbb{C}})^{\tau})_e=\operatorname{Int}(\mathfrak{g}_{\mathbb{C}}^{\tau}).$$
The third condition follows from part (b) of Proposition 7.20 of \cite{Kn2}. The fourth condition is automatically satisfied since $H_e$ is connected.

The statement of the Theorem for $H_e$ then follows immediately from Theorem 4.1 and Proposition 5.1, and the statement for $H$ follows immediately from the statement for $H_e$. Observe that the symmetric pairs $(O(2n,1),O(k)\times O(2n-k,1))$, $(U(n,1),U(k)\times U(n-k,1))$, and $(\operatorname{Sp}(n,1),\operatorname{Sp}(k)\times \operatorname{Sp}(n-k,1))$ all satisfy the hypothesis of the Theorem. 

It was proven in \cite{OV1} that the discrete series occurring in $\pi|_H$ have finite multiplicities. Our Theorem says that the continuous spectrum has finite multiplicities as well.

\section{The Pair $(U(n,1),U(1)\times U(n-1,1))$}

In this section, we prove the following slightly stronger result for the symmetric pair $(U(n,1),U(1)\times U(n-1,1))$. 

\begin{proposition} Suppose $\pi$ is a discrete series representation of $G=U(n,1)$, and consider the symmetric subgroup $H=U(1)\times U(n-1,1)$. Decompose $$\pi|_{U(1)}\cong \bigoplus_{\psi_m\in \widehat{U(1)}} \psi_m\otimes \tau_m$$
into irreducibles under the action of $U(1)$. Here $\tau_m$ is a unitary representation of $U(n-1,1)$. Next, write
$$\tau_m=\int_{\sigma(\delta,\nu)\in \widehat{U(n-1,1)}_{\text{temp.}}} m(\sigma(\delta,\nu),\tau_m)\sigma(\delta,\nu)$$
as a direct integral of irreducible representations of $U(n-1,1)$ (See Section 3 for notation). Then $m(\sigma(\delta,\nu),\tau_m)\neq 0$ for finitely many elliptic parameters $\delta$ and $m(\sigma(\delta,\nu), \tau_m)<\infty$ for all $\sigma,\nu$. In particular, at most finitely many discrete series occur in $\tau_m$. 
\end{proposition}

The proof of this result will require some notation. First, recall the bounded domain model for $U(n,1)/U(n)\times U(1)$. Let $G=U(n,1)$, let $T\subset U(n,1)$ be the diagonal torus, and let $K=U(n)\times U(1)\subset G$ be the maximal compact subgroup of $G$. Let $G_{\mathbb{C}}=\operatorname{GL}(n+1,\mathbb{C})$ denote the complexification of $G$, and let $K_{\mathbb{C}}=\operatorname{GL}(n,\mathbb{C})\times \operatorname{GL}(1,\mathbb{C})$ denote the complexification of $K$. Let 
$$P_{\mathbb{C}}^+=\left\{\left(\begin{matrix} I & B\\ 0 & 1\end{matrix}\right)\Big|\ B\in \mathbb{C}^n\right\}\subset \operatorname{GL}(n+1,\mathbb{C})$$
and let 
$$P_{\mathbb{C}}^-=\left\{\left(\begin{matrix} I & 0\\ C & 1\end{matrix}\right)\Big|\ C\in \mathbb{C}^n\right\}\subset \operatorname{GL}(n+1,\mathbb{C}).$$
Then the product map
$$P_{\mathbb{C}}^+\times K_{\mathbb{C}}\times P_{\mathbb{C}}^-\hookrightarrow \operatorname{GL}(n+1,\mathbb{C})$$
is a diffeomorphism onto an open subset of $G_{\mathbb{C}}$, which contains $G=\operatorname{U}(n,1)$.

Suppose $$g=\left(\begin{matrix} A & B\\ C & d\end{matrix}\right)\in U(n,1)$$
with $A\in M(n,\mathbb{C})$ an $n$ by $n$ matrix, $B$ an $n$ by $1$ matrix, $C$ a $1$ by $n$ matrix, and $d$ a single entry.
Then we may decompose
$$g=\left(\begin{matrix} A & B\\ C & d\end{matrix}\right)=\left(\begin{matrix} I & Bd^{-1}\\ 0 & 1\end{matrix}\right)\left(\begin{matrix} A-Bd^{-1}C & 0\\ 0 & d \end{matrix}\right)\left(\begin{matrix} I & 0\\ d^{-1}C & 1\end{matrix}\right)$$
with respect to the above decomposition. We will call the first term $p_{\mathbb{C}}^+(g)$ and the second term $\kappa_{\mathbb{C}}(g)$. One checks that the set of possible $Bd^{-1}$ that can arise in this way is precisely the set of points in the unit disc, $D_n$, in $\mathbb{C}^n$. We define an action of $U(n,1)$ on $D_n$ by thinking of $D_n\subset P^+_{\mathbb{C}}$ and taking the $P^+_{\mathbb{C}}$ part of $g\cdot z$ for $g\in U(n,1)$ and $z\in D_n$. Written out this action looks like

$$\left(\begin{matrix} A & B\\ C & d\end{matrix}\right)\cdot Z=\frac{AZ+B}{CZ+d}.$$

One checks that the isotropy group at zero is $K=U(n)\times U(1)$; thus, we have an identification $U(n,1)/(U(n)\times U(1))\cong D_n$ (For more, see for instance page 152 of \cite{Kn}).

Let $\lambda\in \widehat{T}=\operatorname{Hom}(T,\mathbb{S}^1)$ be a regular, unitary character of $T$. Given $\mu\in \mathfrak{t}^*$, define $H_{\mu}\in \mathfrak{t}_{\mathbb{C}}$ by $\mu(H)=B(H,H_{\mu})$ for all $H\in \mathfrak{t}$. Then define an inner product on $\mathfrak{t}^*$ by $(\mu_1,\mu_2)=B(H_{\mu_1},H_{\mu_2})$ for $\mu_1,\mu_2\in \mathfrak{t}^*$. Let $\Delta$ be the set of roots of $\mathfrak{g}_{\mathbb{C}}$ with respect to $\mathfrak{t}_{\mathbb{C}}$. Define a set of positive roots,
$$\Delta^+=\{\alpha\in \Delta|\ (\alpha, \lambda)>0\}.$$

Let $\Delta_c^+$ denote the compact, imaginary roots in $\Delta^+$. The $G$ invariant complex structure on $G/K$ gives rise to a $K$-equivariant splitting
$$T_e(G/K)\otimes \mathbb{C}\cong T^{(1,0)}_e(G/K)\oplus T^{(0,1)}_e(G/K)$$
that is preserved under the bracket. Note that $T^{(1,0)}_e(G/K)\subset \mathfrak{g}_{\mathbb{C}}/\mathfrak{k}_{\mathbb{C}}$ is of the form $\mathfrak{b}_{\mathbb{C}}/\mathfrak{k}_{\mathbb{C}}$ where $\mathfrak{b}_{\mathbb{C}}\subset \mathfrak{g}_{\mathbb{C}}$ is a Lie subalgebra containing $\mathfrak{k}_{\mathbb{C}}$. The noncompact, imaginary root spaces contained in $\mathfrak{b}_{\mathbb{C}}$ determine a choice of positive noncompact imaginary roots, $\Phi_n^+$. Let $\Phi^+$ be a choice of positive roots with $\Phi^+\cap \Delta_n=\Phi_n^+$ and $\Phi_c^+=\Delta_c^+$. Define 
$$\widetilde{\rho}=\frac{1}{2}\sum_{\alpha\in \Phi_n^+} \alpha.$$

Let $(\xi_{\lambda-\widetilde{\rho}},V_{\lambda-\widetilde{\rho}})$ be the irreducible representations of $K=U(n)\times U(1)$ with highest weight $\lambda-\widetilde{\rho}$. Let $\mathcal{V}_{\lambda-\widetilde{\rho}}$ be the corresponding $G$-equivariant, holomorphic vector bundle on $G/K$. The discrete series with Harish-Chandra parameter $\lambda$ is isomorphic to 
$$L^2(G/K,\mathcal{V}_{\lambda-\widetilde{\rho}}\otimes T^{(0,q_{\lambda})}(G/K))_{\overline{\Box}}$$
where $\overline{\Box}=\overline{\partial} \overline{\partial}^*+\overline{\partial}^*\overline{\partial}$ and $q_{\lambda}=\#\left(\Phi_n^+\cap \Delta^+\right)$. Here the subscript $\overline{\Box}$ means that we are taking the kernel of the elliptic differential operator $\overline{\Box}$. This result is well-known. It is proved for sufficiently regular parameter in \cite{NO}, and Schmid remarks in \cite{S} that the general statement follows from his proof of the Blattner formula for groups of Hermitian type. Although Schmid does not carry out the details of the proof in this paper, they are very similar to his argument for the Paratharasy Dirac operator model on page 138 of \cite{S}.

Note that we have isomorphisms of vector bundles
$$\mathcal{V}_{\lambda-\widetilde{\rho}}=G\times_{K} V_{\lambda-\widetilde{\rho}}\cong GK_{\mathbb{C}}P_{\mathbb{C}}^-\times_{K_{\mathbb{C}}P_{\mathbb{C}}^-} V_{\lambda-\widetilde{\rho}}\cong D_n\times V_{\lambda-\widetilde{\rho}}.$$
In particular, the bundle $\mathcal{V}_{\lambda-\widetilde{\rho}}$ is holomorphically (though not equivariantly) trivial. The group $G=\operatorname{U}(n,1)$ acts on $D_n\times V_{\lambda-\widetilde{\rho}}$ by
$$g_1\cdot(z,v)=(g_1\cdot z,\xi_{\mathbb{C}}(\kappa_{\mathbb{C}}(g_1g))\xi_{\mathbb{C}}(\kappa_{\mathbb{C}}(g))^{-1}v).$$
Here we have chosen $g\in \operatorname{U}(n,1)$ with $p_{\mathbb{C}}^+(g)=z$ and $g_1\cdot z$ is the above action of $\operatorname{U}(n,1)$ on $D_n$. One checks that this definition is independent of the choice of $g$. Above, $\xi_{\mathbb{C}}$ denotes the holomorphic representation of $K_{\mathbb{C}}$ that is the complexification of the representation $\xi_{\lambda-\widetilde{\rho}}$ of $K$.

Let $\underline{V_{\lambda-\widetilde{\rho}}}$ denote the trivial bundle on $D_n$ with fiber $V_{\lambda-\widetilde{\rho}}$. We obtain a holomorphic isomorphism of vector bundles $$\mathcal{V}_{\lambda-\widetilde{\rho}}\otimes T^{(0,q_{\lambda})}(G/K)\cong\underline{V_{\lambda-\widetilde{\rho}}}\otimes T^{(0,q_{\lambda})}D_n$$
by tensoring the above isomorphism with the map on tangent bundles induced by $G/K\cong D_n$. 

Suppose $$F\in L^2(G/K,\mathcal{V}_{\lambda-\widetilde{\rho}}\otimes T^{(0,q_{\lambda})}(G/K))_{\overline{\Box}}\cong L^2(D_n,\underline{V_{\lambda-\widetilde{\rho}}}\otimes T^{(0,q_{\lambda})}D_n)_{\overline{\Box}}.$$
Let $w_1,\ldots,w_{n-1},w_n=z$ be holomorphic coordinates on $D_n\subset \mathbb{C}^n$ induced from the standard coordinates on $\mathbb{C}^n$. Write
$$F=\sum_I F_{0,I} d\overline{w_I}+\sum_J F_{1,J} d\overline{w_J}\wedge d\overline{z}.$$
Here $I\subset \{1,\ldots,n-1\}$ is a multi-index of cardinality $q_{\lambda}$, $J\subset \{1,\ldots,n-1\}$ is a multi-index of cardinality $q_{\lambda}-1$, and $F_{0,I},F_{1,J}:D_n\rightarrow V_{\lambda-\widetilde{\rho}}$ are analytic functions for each $I$ and $J$.

Then $\overline{\partial}F=0$ implies
$$\frac{\partial F_{0,I}}{\partial\overline{z}}=\sum_{i\in I}\epsilon_{i,I}\frac{\partial F_{1,I-i}}{\partial \overline{w_i}}$$
where $\epsilon_{i,I}=\pm 1$ for each $i$.

Next, we must calculate $\overline{\partial}^*$. We will calculate $\overline{\partial}^*$ by using the formula on page 86 of \cite{AV}, $$\overline{\partial}^*=-\# *\overline{\partial} *\#.$$
To use this formula, we must define the operators $\#$, $*$, and then calculate them in local coordinates. Let 
$$*:\ C^{\omega}(D_n,\underline{V_{\lambda-\widetilde{\rho}}}\otimes T^{(0,q_{\lambda})}D_n)\longrightarrow C^{\omega}(D_n,\underline{V_{\lambda-\widetilde{\rho}}}\otimes T^{(0,q_{\lambda})}D_n)$$
be the usual Hodge $*$ operator. Let $m=(m_{ij})$ be the invariant Hermitian metric on $T^*D_n$ associated to the invariant Riemannian metric on $D_n$ (well-defined up to multiplication by a positive scalar) written in the above coordinates, ie $$m=\sum_{i,j=1}^{n} m_{ij} dw_i d\overline{w_j}.$$
Then 
$$m_{ii}=\frac{1-\sum_{k\neq i}|w_k|^2}{(1-|w|^2)^2},\ m_{ij}=\frac{\overline{w_{i}}w_j}{(1-|w|^2)^2}\ \text{if}\ i\neq j.$$

Let $m^{-1}=(m^{ij})$ be the inverse matrix. Applying Cramer's rule, we observe that the diagonal entries of $m^{-1}$ have the constant term $1$ in their Taylor expansions at zero while the off diagonal $ij$ term is of the form  
$$m^{ij}=w_i\overline{w_j}n^{ij}.$$ 
Here $n^{ij}$ is an analytic function on $D_n$ for $i\neq j$.

Note that $m^{-1}\in \operatorname{End}(T^*D_n)$ naturally gives rise to an endomorphism $M_{q_{\lambda}}\in \operatorname{End}(T^{(0,q_{\lambda})}D_n)$. We express $M_{q_{\lambda}}=(M_{q_{\lambda}}^{I,J})$ in the coordinates $d\overline{w_I}$ with $|I|=q_{\lambda}$ a multi-index. The diagonal entries of $M_{q_{\lambda}}$ are analytic functions with constant term one in their Taylor expansions at $w=0$. Suppose $I,J$ are multi-indices with $I\cap J=K$. Then the $I,J$ off diagonal entry of $M_{q_{\lambda}}$ is of the form
$$M_{q_{\lambda}}^{I,J}=w_{I-K}\overline{w_{J-K}}N_{q_{\lambda}}^{I,J}.$$
The function $N_{q_{\lambda}}^{I,J}$ is analytic on $D_n$. 

In coordinates, we obtain 
$$*(F_{0,I}d\overline{w_I})$$
$$=\det(m_{ij})\sum_{I\cap K=L,\ n\in K}F_{0,I}\epsilon_{I,K} w_{I-L}\overline{z}\overline{w_{K-L-n}}N_{q_{\lambda}}^{I,K}d\overline{w_{K^{c}}}$$
$$+\det(m_{ij})\sum_{I\cap K=L,\ n\notin K}F_{0,I}\epsilon_{I,K} M_{q_{\lambda}}^{I,K}d\overline{w_K^{c}}.$$
Here $K^c=\{1,\ldots,n\}-K$ is the complement of $K$ and $\epsilon_{I,K}=\pm 1$.
Similarly, we obtain
$$*(F_{1,J}d\overline{w_J}\wedge d\overline{z})$$
$$=\det(m_{ij})\sum_{K\cap(J\cup n)=L,\ n\notin K} F_{1,J}\epsilon_{1,J,K}M_{q_{\lambda}}^{J\cup n,K}d\overline{w_{K^c}}$$
$$+\det(m_{ij})\sum_{K\cap(J\cup n)=L,\ n\in K} F_{1,J}\epsilon_{1,J,K}M_{q_{\lambda}}^{J\cup n,K}d\overline{w_{K^c}}.$$
Here $\epsilon_{I,J,K}=\pm 1$.

In order to define $\overline{\partial}^*$, we need another operator as well. Let
$$A\in C^{\omega}(D_n,\operatorname{Hom}_{\mathbb{C}-\text{anti}}(\underline{V_{\lambda-\widetilde{\rho}}},\underline{V_{\lambda-\widetilde{\rho}}}^*))$$ be the Hermitian metric on the bundle $\underline{V_{\lambda-\widetilde{\rho}}}$. Here $\operatorname{Hom}_{\mathbb{C}-\text{anti}}(V_{\lambda-\widetilde{\rho}},V_{\lambda-\widetilde{\rho}}^*)$ denotes the space of $\mathbb{C}$ antilinear homomorphisms from $V_{\lambda-\widetilde{\rho}}$ to $V_{\lambda-\widetilde{\rho}}^*$. 
Define $$\#:\ C^{\omega}(D_n,\underline{V_{\lambda-\widetilde{\rho}}}\otimes T^{(0,q_{\lambda})}D_n)\longrightarrow C^{\omega}(D_n,\underline{V_{\lambda-\widetilde{\rho}}}\otimes T^{(0,q_{\lambda})}D_n)$$
by $$F\mapsto \overline{A}\overline{F}.$$
(One might write $\overline{(A\otimes I)F}$ to be more precise).

Following \cite{AV} page 86, we define
$$\overline{\partial}^*=-\# *\overline{\partial} *\#.$$

\begin{lemma} Suppose $F\in C^{\omega}(D_n,\underline{V_{\lambda-\widetilde{\rho}}}\otimes T^{(0,q_{\lambda})}D_n)_{\overline{\Box}}$, assume $F|_{D_{n-1}}=0$, and write $$F=\sum_I F_{0,I} d\overline{w_I}+\sum_J F_{1,J} d\overline{w_J}\wedge d\overline{z}=F_0+F_1.$$
In the Taylor expansion of $F$ at zero in the variables $z=w_n$, $\overline{z}=\overline{w_n}$, the leading term is holomorphic in $z$ in the $F_0$ part and antiholomorphic in $z$ in the $F_1$ part.
\end{lemma}

A more precise statement of the Lemma is as follows. Taylor expand each $F_{0,I}$ and $F_{1,J}$ in the variables $z=w_n$ and $\overline{z}=\overline{w_n}$ at zero, and suppose that the first nonvanishing term in any of the expansions is of order $k>0$. Then $F_{0,I}\sim z^kv_1+\cdots$ with $v_1\in V_{\lambda-\widetilde{\rho}}$ and $F_{1,J}\sim \overline{z}^kv_2+\cdots$ with $v_2\in V_{\lambda-\widetilde{\rho}}$. 

The first part follows from the equation $\overline{\partial}F=0$. Recall that this implied 
$$\frac{\partial F_{0,I}}{\partial\overline{z}}=\sum_{i\in I}\epsilon_{i,I}\frac{\partial F_{1,I-i}}{\partial \overline{w_i}}$$
where $\epsilon_{i,I}=\pm 1$ for each $i$. Choose $k_1+k_2=k-1$, and differentiate the above equation to obtain
$$\frac{\partial^{k_1+k_2+1} F_{0,I}}{\partial z^{k_1}\partial\overline{z}^{k_2+1}}=\sum_{i\in I}\epsilon_{i,I}\frac{\partial^{k_1+k_2+1} F_{1,I-i}}{\partial \overline{w_i}\partial^{k_1}z\partial^{k_2}\overline{z}}.$$
If we plug in $z=\overline{z}=0$ in the right hand side, we get zero, since we know that every $F_{0,J}$ and $F_{1,J}$ vanishes to degree $k-1$ on $D_{n-1}$ by assumption. Thus, the left hand side vanishes on $D_{n-1}$ as well. The only possibility left for a degree $k$ term is $z^k$. Thus, $F_{0,I}\sim z^kv_1+$ terms of higher order.

For the second part, we must use the $\overline{\partial}^*$ operator, which is a bit more complicated. First, observe that $*$ and $\#$ are injective maps; hence, $\overline{\partial}^*F=0$ iff $\overline{\partial}*\#F=0$. Write $F=\sum F_{0,I}d\overline{w_I}+\sum F_{1,J}d\overline{w_J}\wedge d\overline{z}$. Suppose $K\subset \{1,\ldots,n-1\}$ is a multi-index. The coefficient of $d\overline{w_{K^c}}$ in $\overline{\partial}*\#F=0$ is
$$\sum_{I\cap K=L} \epsilon_{I,K\cup n}\frac{\partial}{\partial \overline{z}}\left(\det(m_{ij})\overline{A}\overline{F_{0,I}}\overline{w_{I-L}}zw_{K-L}\overline{N_{q_{\lambda}}^{I,K\cup n}}\right)$$
$$+\sum_{I} \sum_{k\notin K, k\neq n} \epsilon_{I,K\cup k}\frac{\partial}{\partial \overline{w_k}}\left(\det(m_{ij})\overline{A}\overline{F_{0,I}} \overline{M_{q_{\lambda}}^{I,K\cup k}}\right)$$
$$+\sum_{J} \sum_{k\notin K, k\neq n} \epsilon_{1,J,K\cup k}\frac{\partial}{\partial \overline{w_k}}\left(\det(m_{ij})\overline{AF_{1,J}}\overline{M_{q_{\lambda}}^{J\cup n,K\cup k}}\right)$$
$$+\sum_{K\cap J=L} \epsilon_{1,J,K}\frac{\partial}{\partial \overline{z}}\left(\det(m_{ij})\overline{AF_{1,J}}\overline{M_{q_{\lambda}}^{J\cup n,K\cup n}}\right).$$
This follows from the calculation of $*$ in coordinates on the last page together with the definition of $\#$. Since $\overline{\partial}^*F=0$, we know that the above expression is zero for every $K$. Now, choose $k_1+k_2=k-1$, and apply the differential operator $$\frac{\partial^{k_1}}{\partial z^{k_1}}\frac{\partial^{k_2}}{\partial \overline{z}^{k_2}}$$
to the above expression and restrict to $D_{n-1}$. The second and third expressions are zero since $F_{0,I}$ and $F_{1,J}$ vanish to order $k-1$ in $z$ and $\overline{z}$ on $D_{n-1}$. The first expression vanishes because of the extra factor of $z$. If we use one of our $z$ derivatives to differentiate this $z$, then the $F_{1,J}$ vanishes. If we do not differentiate the $z$, then plugging in $z=0$ makes the expression vanish. Thus, we are left with
$$\sum_{K} \epsilon_{1,J,K}\frac{\partial^k}{\partial z^{k_1}\partial \overline{z}^{k_2+1}}\left(\det(m_{ij})\overline{AF_{1,J}}\overline{M_{q_{\lambda}}^{J\cup n,K\cup n}}\right)=0$$
on $D_{n-1}$. Applying the product rule and using the facts that $F_{1,J}$ vanishes to degree $k-1$ on $D_{n-1}$, we obtain
$$\sum_{K} \epsilon_{1,J,K}\overline{A}\left(\frac{\partial^k}{\partial z^{k_1}\partial \overline{z}^{k_2+1}}\overline{F_{1,J}}\right)\det(m_{ij})\overline{M_{q_{\lambda}}^{J\cup n,K\cup n}}=0$$
on $D_{n-1}$. Now, let $\mathcal{M}$ be the matrix that takes a vector $(v_J)$ to the vector $(\det(m_{ij})\sum_K \epsilon_{K,J,1}\overline{M_{q_{\lambda}}^{J\cup n,K\cup n}}v_J)$. At zero, $\mathcal{M}$ is a diagonal matrix with entries $\pm 1$ along the diagonal. Since the entries of $\mathcal{M}$ are analytic, we conclude that $\mathcal{M}$ must be invertible in a neighborhood of zero. In particular, $$(\det(m_{ij})\sum_K \epsilon_{K,J,1}\overline{M_{q_{\lambda}}^{J\cup n,K\cup n}}v_J)=0 \Rightarrow (v_J)=0$$ in a neighborhood of zero. We conclude
$$\overline{A}\left(\frac{\partial^k}{\partial z^{k_1}\partial \overline{z}^{k_2+1}}\overline{F_{1,J}}\right)=0$$
in a neighborhood of zero inside $D_{n-1}$ for every $K$ and every $k_1+k_2=k-1$. But, $\overline{A}$ is invertible everywhere; therefore,
$$\frac{\partial^k}{\partial z^{k_1}\partial \overline{z}^{k_2+1}}\overline{F_{1,J}}=0$$
in a neighborhood of zero in $D_{n-1}$. By analyticity of $F_{1,J}$, this expression is zero on all of $D_{n-1}$, and taking the complex conjugate, we obtain
$$\frac{\partial^k}{\partial \overline{z}^{k_1}\partial z^{k_2+1}}F_{1,J}=0$$
on $D_{n-1}$. We conclude $F_{1,J}\sim \overline{z}^kv_2+$ higher order terms for every $J$. The Lemma has been proven.

\bigskip

Now, we restrict $\pi_{\lambda}$ to $H=U(1)\times U(n-1,1)\subset G=U(n,1)$. Note that $H$ has a two dimensional center. We consider the one dimensional subgroup $$U(1)=\left\{k_{\theta}=\left(\begin{matrix} e^{i\theta} & 0 \\ 0 & I_n\end{matrix}\right)\right\}$$ 
where $I_n$ is the $n$ by $n$ identity matrix and the zeroes are $n$ by $1$ and $1$ by $n$ matrices. 
Note that this subgroup is central in $U(1)\times U(n-1,1)$, but it is not central in $U(n,1)$. Note that the irreducible, unitary characters of $U(1)$ are the characters $\psi_m: e^{i\theta}\mapsto e^{im\theta}$ for $m\in \mathbb{Z}$. We may decompose 
$$\pi_{\lambda}=\bigoplus_{m\in \mathbb{Z}}(\pi_{\lambda})_m\cong \bigoplus_{m\in \mathbb{Z}} \psi_m\otimes \operatorname{Hom}(\psi_m, \pi_{\lambda})=\bigoplus_{m\in \mathbb{Z}} \psi_m\otimes \tau_m$$
under the action of $U(1)$. Observe $\tau_m$ is a representation of $U(n-1,1)$. We identify $\pi_{\lambda}\cong L^2(D_n,\underline{V_{\lambda-\widetilde{\rho}}}\otimes T^{(0,q_{\lambda})}D_n)_{\overline{\Box}}$, and $(\pi_{\lambda})_m$ is therefore a subspace of these $L^2$-sections.

\begin{lemma} Suppose $F\in (\pi_{\lambda})_m\subset L^2(D_n,\underline{V_{\lambda-\widetilde{\rho}}}\otimes T^{(0,q_{\lambda})}D_n)_{\overline{\Box}}$, and write $$F=\sum_I F_{0,I} d\overline{w_I}+\sum_J F_{1,J} d\overline{w_J}\wedge d\overline{z}=F_0+F_1$$ in coordinates. Decompose $V_{\lambda-\widetilde{\rho}}=\bigoplus V_{\lambda-\widetilde{\rho}}^s$ under $U(1)$, and let $F_{0,I}=\sum F_{0,I}^s$ and $F_{1,J}=\sum F_{1,J}^s$ be the corresponding decompositions into components. Now, Taylor expand $F_{0,I}^s$ and $F_{1,J}^s$ in the variables $z$ and $\overline{z}$ at zero.
Then
$$F_{0,I}^s=\sum_{i=\max(0,s-m)}^{\infty} F_{0,I}^{s,i}z^i\overline{z}^{m-s+i},\ F_{1,J}^s=\sum_{i=\max(0,s-m-1)}^{\infty} F_{1,J}^{s,i}z^i\overline{z}^{m-s+1+i}.$$
\end{lemma}

\begin{proof} We begin the proof of the Lemma by calculating the action of $U(1)$ on $\underline{V_{\lambda-\widetilde{\rho}}}\otimes T^{(0,q_{\lambda})}D_n$. 
We compute
$$k_{\theta}\cdot (z,w_1,\ldots,w_{n-1},v\otimes d\overline{w_I})$$
$$=(e^{i\theta}z,w_1,\ldots,w_{n-1},\xi_{\mathbb{C}}(\kappa_{\mathbb{C}}(k_{\theta}g))\xi_{\mathbb{C}}(\kappa_{\mathbb{C}}(g))^{-1}v\otimes d\overline{w_I})$$
$$=(e^{i\theta}z,w_1,\ldots,w_{n-1},\xi_{\lambda-\widetilde{\rho}}(k_{\theta})v\otimes d\overline{w_I}).$$
Here $g\in U(n,1)$ is chosen so that $g\cdot 0=(z,w_1,\ldots,w_{n-1})$ and we have used $\xi_{\mathbb{C}}(\kappa_{\mathbb{C}}(k_{\theta}g))=\xi_{\lambda-\widetilde{\rho}}(k_{\theta})\xi_{\mathbb{C}}(\kappa_{\mathbb{C}}(g))$. Similarly, we have
$$k_{\theta}\cdot (z,w_1,\ldots,w_{n-1},v\otimes d{\overline{w_J}}\wedge d\overline{z})$$
$$=(e^{i\theta}z,w_1,\ldots,w_{n-1},\xi_{\mathbb{C}}(\kappa_{\mathbb{C}}(k_{\theta}g))\xi_{\mathbb{C}}(\kappa_{\mathbb{C}}(g))^{-1}v\otimes e^{-i\theta}d\overline{w_I}\wedge d\overline{z})$$
$$=(e^{i\theta}z,w_1,\ldots,w_{n-1},\xi_{\lambda-\widetilde{\rho}}(k_{\theta})v\otimes e^{-i\theta}d\overline{w_I}\wedge d\overline{z}).$$
Now, decompose $V_{\lambda-\widetilde{\rho}}=\oplus V_{\lambda-\widetilde{\rho}}^s$ as in the statement of the Lemma, and similarly decompose $F_{0,I}d\overline{w_I}=\sum F_{0,I}^sd\overline{w_I}$. Fix $s$ and Taylor expand 
$$F_{0,I}^sd\overline{w_I}=\sum_{j,l}F_{0,I}^{s,j,l}z^j\overline{z}^l$$
at zero in $z$ and $\overline{z}$. Here $F_{0,I}^{s,j,l}$ is an analytic function on $D_{n-1}$. Then we have
$$k_{\theta}\cdot \left(\sum F_{0,I}^{s,j,l}z^j\overline{z}^ld\overline{w_I}\right)=\xi_{\lambda-\widetilde{\rho}}(k_{\theta})\sum F_{0,I}^{s,j,l}e^{i(l-j)\theta}z^j\overline{z}^ld\overline{w_I}$$
$$=\sum F_{0,I}^{s,j,l}e^{i(s+l-j)\theta}z^j\overline{z}^ld\overline{w_I}.$$
Now, we assume $F\in (\pi_{\lambda})_m$. Thus, 
$$k_{\theta}\cdot \left(\sum F_{0,I}^{s,j,l}z^j\overline{z}^ld\overline{w_I}\right)=e^{im\theta}\left(\sum F_{0,I}^{s,j,l}z^j\overline{z}^ld\overline{w_I}\right).$$
In particular, $F_{0,I}^{s,j,l}\neq 0$ implies $s+l-j=m$ or $l-j=m-s$. Thus, the Taylor expansion of $F_{0,I}$ in the variables $z$ and $\overline{z}$ at zero is of the form
$$F_{0,I}=\sum F_{0,I}^{s,j,j+m-s}z^j\overline{z}^{j+m-s}d\overline{w_I}.$$

Now, consider the case $$F^s_{1,J}d\overline{w_J}\wedge d\overline{z}=\sum F_{1,J}^{s,j,l}z^j\overline{z}^ld\overline{w_J}\wedge d\overline{z}.$$
Then we have
$$k_{\theta}\cdot \left(\sum F_{1,J}^{s,j,l}z^j\overline{z}^ld\overline{w_I}\wedge d\overline{z}\right)=\xi_{\lambda-\widetilde{\rho}}(k_{\theta})\sum F_{0,I}^{s,j,l}e^{i(l-j)\theta}z^j\overline{z}^ld\overline{w_I}$$
$$=\sum F_{0,I}^{s,j,l}e^{i(s+l-j-1)\theta}z^j\overline{z}^ld\overline{w_I}.$$
And since $F\in (\pi_{\lambda})_m$, we must have $s+l-j-1=m$ or $l-j=1+m-s$. The lemma follows.
\end{proof}

Now, we have all of the Lemmas we need to prove the Proposition at the beginning of this section.

\begin{proof} We write $H=U(1)\times U(n-1,1)$, and we decompose $\pi_{\lambda}=\sum_m (\pi_{\lambda})_m$ as a representation of $U(1)$. Let $\tau$ be the involution of $G$ with $G^{\tau}=H$, and let $\theta$ be the Cartan involution of $G$ with $G^{\theta}=K$. Differentiate $\tau$ and $\theta$ to obtain involutions of the Lie algebra $\mathfrak{g}$; we will abuse notation and call these involutions $\tau$ and $\theta$ as well. Let $\mathfrak{g}=\mathfrak{k}\oplus \mathfrak{p}$ be the decomposition of $\mathfrak{g}$ into $+1$ and $-1$ eigenspaces of $\theta$, and let $\mathfrak{g}=\mathfrak{h}\oplus \mathfrak{q}$ be the decomposition of $\mathfrak{g}$ into $+1$ and $-1$ eigenspaces of $\mathfrak{g}$. Then we have the simultaneous decomposition
$$\mathfrak{g}=\mathfrak{k}\cap \mathfrak{h}\oplus \mathfrak{k}\cap \mathfrak{q}\oplus \mathfrak{p}\cap \mathfrak{h}\oplus \mathfrak{p}\cap \mathfrak{q}.$$

Let $s$ be a highest weight of $V_{\lambda-\widetilde{\rho}}|_{U(1)}$. Consider the $H$-equivariant map introduced in Proposition 2.1 and \cite{OV1}, \cite{OV2} 
$$r_l: L^2(D_n,\underline{V_{\lambda-\widetilde{\rho}}}\otimes T^{(0,q_{\lambda})}D_n)_{\overline{\Box}}\longrightarrow$$ $$L^2(D_{n-1},\operatorname{Hom}_{\mathbb{C}}(S_l(\mathfrak{p}\cap \mathfrak{q}),(\underline{V_{\lambda-\widetilde{\rho}}}\otimes T^{(0,q_{\lambda})}D_n)|_{D_{n-1}}))$$
where $l=|m-s|+1$.
We claim that the restriction of $r_l$ to $(\pi_{\lambda})_m$ is injective. To see this, let $F\in (\pi_{\lambda})_m$, and write $F=\sum F_{0,I}d\overline{w_I}+\sum F_{1,J} d\overline{w_J}\wedge d\overline{z}$. By Lemma 6.2, we know that the leading term of $F$ is holomorphic in $F_0$ and antiholomorphic in $F_1$. And by Lemma 6.3, we know that the terms in the Taylor expansions of $F_{0,I}$ and $F_{1,J}$ at zero in $z$ and $\overline{z}$ must be of the form $$z^i\overline{z}^{i+m-s}\ \text{or}\ z^i\overline{z}^{i+m-s+1}.$$
In particular, a holomorphic or antiholomorphic term cannot be of degree greater than $|m-s|+1$. Thus, $F$ cannot vanish to degree greater than $|m-s|+1$ at zero, and our map is indeed injective. As in section one, we utilize the Polar decomposition to obtain a unitary map
$$U_l:(\pi_{\lambda})_m\longrightarrow L^2(D_{n-1},\operatorname{Hom}_{\mathbb{C}}(S_l(\mathfrak{p}\cap \mathfrak{q}),(\underline{V_{\lambda-\widetilde{\rho}}}\otimes T^{(0,q_{\lambda})}D_n)|_{D_{n-1}})).$$
Now, suppose $G$ is a reductive Lie group of Harish-Chandra class, suppose $K\subset G$ is a maximal compact subgroup, suppose $(\eta,V)$ is a finite dimensional representation of $K$, and let $\mathcal{V}\rightarrow G/K$ be the corresponding $G$-equivariant vector bundle. Then $\sigma(\delta,\nu)$ occurs in $L^2(G/K,\mathcal{V})$ iff one of the irreducible constituents of $(\eta,V)$ occurs as a $K$ type of $\sigma(\delta,\nu)$ by Frobenius reciprocity. In addition, all such representations occur with finite multiplicity. One observes that there are a finite number of elliptic parameters $\delta$ such that $\sigma(\delta,\nu)$ contains a fixed $K$ type. Thus, $L^2(G/K,\mathcal{V})$ contains $\sigma(\delta,\nu)$ for a finite number of elliptic parameters $\delta$ and each $\sigma(\delta,\nu)$ occurs with finite multiplicity. 
Finally, since we have an injective, unitary map of $(\pi_{\lambda})_m$ into a finite number of such spaces, the Theorem follows.
\end{proof}

\section{Acknowledgments} The first author would like to thank all of the participants and organizers of the conference ``Branching Problems for Unitary Representations'' at the Max Planck Institute in 2011. Some conversations at that conference, especially ones with Jorge Vargas and Michel Duflo, contributed to the motivation behind beginning this project. The first author would also like to thank Toshiyuki Kobayashi for giving a course at Harvard in 2008 that first introduced him to branching problems.

\bibliographystyle{amsplain}

\end{document}